\documentclass[12pt]{article}
\usepackage{graphicx,subfigure}
\usepackage{amsmath,amsthm,amssymb,enumerate}
\usepackage{euscript,mathrsfs}
\usepackage{color}
\usepackage{dsfont}
\usepackage{url}
\usepackage{bm}
\usepackage[left=2cm,right=2cm,top=3.5cm,bottom=3.5cm]{geometry}
\usepackage{color}
\usepackage[framemethod=tikz]{mdframed}
\allowdisplaybreaks

\usepackage{esint}
\usepackage{soul}

\catcode`\@=11 \@addtoreset{equation}{section}

\catcode`\@=12

\newtheorem{Theorem}{Theorem}[section]
\newtheorem{Proposition}[Theorem]{Proposition}
\newtheorem{Lemma}[Theorem]{Lemma}
\newtheorem{Corollary}[Theorem]{Corollary}

\theoremstyle{definition}
\newtheorem{Definition}[Theorem]{Definition}

\newtheorem{Remark}[Theorem]{Remark}

\newcommand{\bTheorem}[1]{
	\begin{Theorem} \label{T#1} }
	\newcommand{\eT}{\end{Theorem}}

\newcommand{\bProposition}[1]{
	\begin{Proposition} \label{P#1}}
	\newcommand{\eP}{\end{Proposition}}

\newcommand{\bLemma}[1]{
	\begin{Lemma} \label{L#1} }
	\newcommand{\eL}{\end{Lemma}}

\newcommand{\bCorollary}[1]{
	\begin{Corollary} \label{C#1} }
	\newcommand{\eC}{\end{Corollary}}

\newcommand{\bRemark}[1]{
	\begin{Remark} \label{R#1} }
	\newcommand{\eR}{\end{Remark}}

\newcommand{\bDefinition}[1]{
	\begin{Definition} \label{D#1} }
	\newcommand{\eD}{\end{Definition}}

\newcommand{\tvS}{\widetilde{S}}

\newcommand{\TD}{\mathbb{T}^2}

\newcommand{\avintO}[1]{\fint_{\Omega} #1 \dx}

\newcommand{\intTS}[1]{\int_{\mathbb{R} \times \mathbb{T}^1} #1 \ \dx}

\newcommand{\Td}{\mathbb{T}^d}

\newcommand{\intTd}[1]{\int_{\mathbb{T}^d} #1 \ \dx}

\newcommand{\vme}{\vm_\ep}

\newcommand{\tvm}{\widetilde{\vc{m}}}
\renewcommand{\tvm}{\widetilde{\bm{m}}}

\newcommand{\bfphi}{\boldsymbol{\varphi}}

\newcommand{\bFormula}[1]{
	\begin{equation} \label{#1}}
	\newcommand{\eF}{\end{equation}}

\newcommand{\Ov}[1]{\overline{#1}}

\newcommand{\toW}{\stackrel{\mathcal{W}}{\to}}

\newcommand{\vr}{\varrho}
\newcommand{\vre}{\vr_\ep}

\newcommand{\tvr}{\wtilde \vr}

\newcommand{\vt}{\vartheta}
\newcommand{\vu}{\bm{u}}
\newcommand{\vm}{\bm{m}}

\newcommand{\vc}[1]{{\bf #1}}

\newcommand{\Div}{{\rm div}_x}
\newcommand{\Grad}{\nabla_x}

\newcommand{\dx}{\,{\rm d} {x}}

\newcommand{\dt}{\,{\rm d} t }

\newcommand{\vU}{\bm{U}}

\newcommand{\intO}[1]{\int_{\Omega} #1 \ \dx}

\newcommand{\D}{{\rm d}}

\newcommand{\ep}{\varepsilon}

\newcommand{\R}{\mathbb{R}}

\newcommand{\br}{ \nonumber \\ }

\def\softd{{\leavevmode\setbox1=\hbox{d}%
		\hbox to 1.05\wd1{d\kern-0.4ex{\char039}\hss}}}
\definecolor{Cgrey}{rgb}{0.85,0.85,0.85}
\definecolor{Cblue}{rgb}{0.50,0.85,0.85}
\definecolor{Cred}{rgb}{1,0,0}
\definecolor{fancy}{rgb}{0.10,0.85,0.10}
\definecolor{amaranth}{rgb}{0.9, 0.17, 0.31}

\newcommand\Cbox[2]{%
	\newbox\contentbox%
	\newbox\bkgdbox%
	\setbox\contentbox\hbox to \hsize{%
		\vtop{
			\kern\columnsep
			\hbox to \hsize{%
				\kern\columnsep%
				\advance\hsize by -2\columnsep%
				\setlength{\textwidth}{\hsize}%
				\vbox{
					\parskip=\baselineskip
					\parindent=0bp
					#2
				}%
				\kern\columnsep%
			}%
			\kern\columnsep%
		}%
	}%
	\setbox\bkgdbox\vbox{
		\color{#1}
		\hrule width  \wd\contentbox %
		height \ht\contentbox %
		depth  \dp\contentbox
		\color{black}
	}%
	\wd\bkgdbox=0bp%
	\vbox{\hbox to \hsize{\box\bkgdbox\box\contentbox}}%
	\vskip\baselineskip%
}

\mdfdefinestyle{MyFrame}{%
	linecolor=black,
	outerlinewidth=1pt,
	roundcorner=5pt,
	innertopmargin=\baselineskip,
	innerbottommargin=\baselineskip,
	innerrightmargin=10pt,
	innerleftmargin=10pt,
	backgroundcolor=white!20!white}

\newcommand{\wtilde}{\widetilde}

\allowdisplaybreaks

\begin{document}


\title{\bf The Euler system of gas dynamics}

\author{Eduard Feireisl
	\thanks{The work of E.F. was partially supported by the
		Czech Sciences Foundation (GA\v CR), Grant Agreement
		24--11034S. The Institute of Mathematics of the Academy of Sciences of
		the Czech Republic is supported by RVO:67985840.
		E.F. is a member of the Ne\v cas Center for Mathematical Modelling and Mercator Fellow in SPP 2410 ``Hyperbolic Balance Laws: Complexity, Scales and Randomness".}
}

\date{}

\maketitle

\centerline{$^*$Institute of Mathematics of the Academy of Sciences of the Czech Republic}
\centerline{\v Zitn\' a 25, CZ-115 67 Praha 1, Czech Republic}
\centerline{feireisl@math.cas.cz}

\begin{abstract}
	
This is a survey highlighting several recent results concerning well/ill posedness of the Euler system of gas dynamics. 
Solutions of the system are identified as limits of consistent approximations generated either by physically more complex 
problems, notably the Navier-Stokes-Fourier system, or by the approximate schemes in numerical experiments. The role 
of the fundamental principles encoded in the First and Second law of thermodynamics in identifying a unique physically admissible 
solution is examined. 
 	
\end{abstract}


{\small


\medbreak

\tableofcontents

}

\section{Introduction}
\label{i}

The \emph{Euler system of gas dynamics} is a mathematical model describing the time evolution of an ideal gas in the framework of continuum 
mechanics. From the mathematics viewpoint, the problem represents an iconic example of a system of non--linear conservation laws, see 
Part IV of the monograph by Benzoni-Gavage and Serre \cite{BenSer} or Dafermos \cite{D4a}. The system consists of three field equations 
reflecting the underlying physical principles of 
 \begin{align}
\bullet \ \mbox{\bf mass conservation:} \hskip 7.8 cm	\partial_t \vr + \Div \vm &= 0, \label{i1}\\
\bullet \ \mbox{\bf momentum balance:} \hskip 4.3 cm	\partial_t \vm + \Div \left( \frac{\vm \otimes \vm}{\vr} \right) +
	\Grad p &= 0, \label{i2} \\
\bullet \ \mbox{\bf energy balance:} \hskip 6.2 cm	\partial_t E + \Div \left[ \Big( E + p \Big) \frac{\vm}{\vr} \right] &= 0,
	\label{i3}	
\end{align}
where $\vr = \vr(t,x)$ is the \emph{mass density}, $\vm = \vm(t,x) = \vr \vu (t,x)$ the \emph{momentum} with the corresponding bulk velocity $\vu$,
$p$ the pressure, and $E$ the \emph{total energy}. For the sake of simplicity, we suppose the gas occupies a bounded domain $\Omega \subset \R^d$, $d=1,2,3$, 
with an impermeable rigid boundary, 
\begin{equation} \label{i4}
	\vm \cdot \bm{n}|_{\partial \Omega} = 0,\ \bm{n} \ \mbox{the outer normal vector to}\ \partial \Omega.
\end{equation}
In this study, we focus mostly on the physically relevant multi--dimensional setting $d = 3$.

\subsection{The Second law of thermodynamics, entropy}
\label{se}

The Euler system \eqref{i1}--\eqref{i3} supplemented with the boundary condition \eqref{i4} is \emph{thermodynamically closed}, meaning 
neither mass nor energy are exchanged with the ``outer world''. The asymptotic behaviour of closed systems 
is characterized in the most flagrant way by the celebrated statement of Clausius: 
\begin{quotation}
	
	\emph{
		The energy of the world is constant; its entropy tends to a maximum.}
	
\end{quotation}

While the time evolution of the energy is governed by equation \eqref{i3}, the entropy can be identified only through a suitable 
\emph{equation of state} relating the pressure $p$ and the \emph{internal energy} 
\[
\vr e = E - \frac{1}{2} \frac{|\vm|^2}{\vr}
\]  
to the density $\vr$ and the \emph{absolute temperature} $\vt$. The existence of \emph{entropy} $s$ follows by imposing 
Gibbs' law
\begin{equation} \label{i5} 
\vt D s = D e + p D \left( \frac{1}{\vr} \right). 
\end{equation}
Formally, meaning under the assumption of smoothness of all quantities in question, the entropy $s$ satisfies a transport equation 
\begin{equation} \label{i6}
\partial_t (\vr s) + \Div (s \vm) = 0 \ \mbox{or}\ \partial_t s + \vu \cdot \Grad s = 0.
\end{equation}
In particular, the entropy remains constant if constant initially, which gives rise to a reduced \emph{isentropic} Euler system. Here, we focus 
on the general case with non-constant entropy.

Note that \eqref{i6} is clearly at odds with Clausius' statement quoted above as the entropy obeying 
\eqref{i6} is a \emph{conserved quantity}. In the context of gas dynamics, however, solutions of the Euler system 
are known to develop singularities - the so--called shock waves - in a finite time for a fairly general class of the initial states. 
Accordingly, \emph{global-in-time} solutions, as long as they exist, satisfy the Euler system in the weak (distributional) sense. The entropy 
equation \eqref{i6} is relaxed to the inequality
\[
\partial_t (\vr s) + \Div (s \vm) \geq 0,  
\]
meaning the entropy is produced in accordance with the Second law of thermodynamics. 
This can be viewed as a ghost effect of dissipative terms in the vanishing viscosity limit, where the Euler system is understood 
as an asymptotic state of the Navier--Stokes--Fourier system with vanishing dissipation.

Introducing 
the \emph{total entropy} $S = \vr s$ we write the entropy inequality in the form of
\begin{equation} \label{i7}
	\partial_t S + \Div \left( S \frac{\vm}{\vr} \right)  \geq 0
\end{equation}	
usually interpreted as an \emph{admissibility criterion} in the class of weak solutions of the Euler system. 
Moreover, by means of Schwartz representation theorem, we may interpret \eqref{i7} as 
\begin{equation} \label{i8}
	\bullet \ \mbox{\bf entropy inequality:} \hskip 6 cm	\partial_t S + \Div \left( S \frac{\vm}{\vr} \right)= \sigma,
\end{equation}	 
with the \emph{entropy production rate} $\sigma$ -- a non--negative measure. 

We consider a \emph{polytropic equation of state}:
\begin{equation} \label{i9}
	p = (\gamma - 1) \vr e, \ \mbox{with the adiabatic exponent}\ \gamma > 1.
\end{equation}
Apparently, 
the equation of state \eqref{i9} is \emph{incomplete}, in particular, the temperature $\vt$ is not uniquely determined. Indeed Gibbs' equation 
\eqref{i5}, together with \eqref{i9}, yield 
\begin{equation} \label{i10} 
	p(\vr, \vt) = \vt^{\frac{\gamma}{\gamma - 1}} P \left( \frac{\vr}{ \vt^{\frac{1}{\gamma - 1}}} \right), 
\end{equation}	
and, 
\begin{align} 
	e(\vr, \vt) &= \frac{\vt}{\gamma - 1}\frac{\vt^{\frac{1}{\gamma - 1}}}{\vr} P \left( \frac{\vr}{ \vt^{\frac{1}{\gamma - 1}}} \right), \br 
	s(\vr, \vt) &= S \left( \frac{\vr}{ \vt^{\frac{1}{\gamma - 1}}} \right),\ S'(Z) = - \frac{1}{\gamma - 1} \frac{\gamma P(Z) - P'(Z) Z }{Z^2}, 
	\label{i11}
\end{align}
for an arbitrary function $P$. Thus the absolute temperature $\vt$ is determined by $\vr$ and $e$ modulo the function $P$, see Cowperthwaite \cite{Cowp}, M\" uller and Ruggeri \cite{MURU}, or \cite[Chapters 2,3]{FeNo6A}. To fix this problem, we consider $P(Z) = Z$ obtaining the standard 
\emph{Boyle--Mariotte law} 
\begin{equation} \label{i12}
p(\vr, \vt) = \vr \vt,\ e = c_v \vt,\ c_v = \frac{1}{\gamma - 1},\ s(\vr, \vt) = c_v \log (\vt) - \log(\vr).
\end{equation}

\subsection{State--of--the art}

Let us briefly summarize the available results concerning well--posedness of the Euler system. 

\subsubsection{Strong/classical solutions}

As is well known, the Euler system is well posed, locally in time, for sufficiently regular initial data. The relevant result 
was obtained by Schochet \cite[Theorem 1]{SCHO1}.

\begin{Theorem}[{\bf Short time existence}] \label{Ti1}
	Suppose that
	\begin{itemize}
		\item the thermodynamic functions $p = p(\vr, \vt)$, $e = e(\vr, \vt)$, $s = s(\vr, \vt)$ are three times continuously
		differentiable for $\vr > 0$, $\vt > 0$, satisfy Gibbs' law \eqref{i5}, and 
		\[
		\frac{\partial p(\vr, \vt) }{\partial \vr} > 0,\ \frac{\partial e(\vr, \vt) }{\partial \vt} > 0 \ \mbox{for all}\ \vr > 0, \vt > 0.
		\]
		\item
		$\Omega \subset \R^3$ is a bounded domain with $C^\infty-$boundary;
		\item the initial data belong to the class
		\[
		\vr_0, \ \vt_0 \in W^{3,2}(\Omega),\ \vu_0 \in W^{3,2}(\Omega; \R^3),\ \vr_0 > 0,\ \vt_0 > 0 \ \mbox{in}\ \Ov{\Omega};
		\]
		\item the compatibility conditions
		\[
		\partial^k_t \vu_0 \cdot \vc{n}|_{\partial \Omega} = 0
		\]
		hold for $k=0,1,2$.
		
	\end{itemize}
	
	Then there exists $T > 0$ such that
	the Euler system \eqref{i1} -- \eqref{i3}, with the boundary condition \eqref{i4}, and the initial data
	\[
	\vr(0, \cdot) = \vr_0,\ \vu(0, \cdot) = \vu_0,\ \vt(0, \cdot) = \vt_0,
	\]
	admits a classical solution in $(0,T) \times \Omega$,
	\[
	\vr(t, \cdot), \ \vt(t, \cdot) > 0 \ \mbox{in}\ \Ov{\Omega} \ \mbox{for}\ t \in [0,T).
	\]
	
\end{Theorem}

By classical solution we mean that all functions $\vr$, $\vt$, and $\vu$ are continuously differentiable up to the boundary
and the equations hold pointwise in $(0,T) \times \Omega$. As a matter of fact, the solutions are constructed in the
Sobolev class
\[
\cap_{k = 0}^3 C^{3 - k}([0,T]; W^{k,2}(\Omega)),
\]
whence the desired regularity follows from the Sobolev embedding $W^{2,2}(\Omega) \hookrightarrow C(\Ov{\Omega})$, $\Omega
\subset \R^3$. 

\subsubsection{Weak solutions}
\label{ws}

It is well known that the life span of classical/strong solutions may be finite for a large class of initial data, see e.g. the monograph by Smoller 
\cite{SMO}. Thus if the Euler system is to be retained as a mathematical model of the gas in a large time lap, the class of solutions must be extended.  
In view of the character of the available results we want to discuss, we consider the space periodic boundary conditions replacing $\Omega$ by the flat torus
\[
\Omega = \mathbb{T}^d = \left\{ (x_1,\dots, x_d) \ \Big| \ x_i \in [0,1]|_{\{0;1\}},\ i = 1, \dots, d \right\},\ d=1, 2,3.
\]

Adopting the equation of state \eqref{i12} we say
a triple $(\vr, \vt, \vu)$ is \emph{weak solution} of the Euler system in $[0,T) \times \mathbb{T}^d$ emanating from the initial 
data $(\vr_0, \vt_0, \vu_0)$ if the following holds:
\begin{equation} \label{i13}
	\int_0^T \intTd{ \Big( \vr \partial_t \varphi + \vr \vu \cdot \Grad \varphi \Big) } \dt = - \intTd{ \vr_0 \varphi (0, \cdot) }
\end{equation}
for any $\varphi \in C^1_c([0,T) \times \Td)$; 
\begin{equation} \label{i14}
	\int_0^T \intTd{ \Big( \vr \vu \cdot \partial_t \bfphi + \vr \vu \otimes \vu: \Grad \bfphi + p(\vr, \vt) \Div \bfphi \Big) } \dt = - \intTd{ \vr_0 \vu_0 \cdot \bfphi (0, \cdot) }	
\end{equation}
for any $\bfphi \in C^1_c([0,T) \times \Td; \R^d)$;
\begin{align} 
	\int_0^T &\intTd{ \left( \left( \frac{1}{2} \vr |\vu|^2 + \vr e(\vr, \vt) \right) \partial_t \varphi + 
		\left( \frac{1}{2} \vr |\vu|^2 + \vr e(\vr, \vt) + p(\vr, \vt) \right) \vu  \cdot \Grad \varphi \right) } \dt \br &\quad = - \intTd{
		\left( \frac{1}{2} \vr_0 |\vu_0|^2 + \vr_0 e(\vr_0, \vt_0) \right) \varphi (0, \cdot) }
	\label{i15}
\end{align}
for any $\varphi \in C^1_c([0,T) \times \Td)$. 

In addition, we say that a weak solutions is (entropy) admissible, if the entropy inequality
\begin{equation} \label{i16}	
	\int_0^T \intTd{ \Big( \vr s(\vr, \vt) \partial_t \varphi + \vr s(\vr, \vt) \vu \cdot \Grad \varphi \Big) } \dt \leq - 
	\intTd{ \vr_0 s(\vr_0, \vt_0) \varphi (0, \cdot) }	
\end{equation}
holds for any  $\varphi \in C^1_c([0,T) \times \Td)$, $\varphi \geq 0$.

The weak solutions (without any admissibility condition like \eqref{i16}) are not uniquely determined by the 
initial data, see Dafermos \cite{D4a},  Smoller \cite{SMO}. Recently, the method of convex integration 
developed in the context of Euler equations in the seminal work of De Lellis and Sz\' ekelihidi \cite{DelSze13}, \cite{DelSze3}, 
revealed a number of rather disturbing facts concerning well posedness of the Euler system even in the class of \emph{entropy admissible} 
weak solutions. 

Let us start with the classical Riemann problem on an infinite strip 
 \[
\mathbb{R} \times \mathbb{T}^1 = \left\{ (x_1,x_2) \ \Big|\ x_1 \in \mathbb{R},\ x_2 \in [0,1]|_{\{0;1\}} \right\}, 
\]
with the initial data determined by constant vectors
\begin{equation} \label{i17}
	(\vr_\ell, \vt_\ell, \vu_\ell),\ (\vr_r, \vt_r, \vu_r) \in (0,\infty) \times (0, \infty) \times \mathbb{R}^2. 
\end{equation}	
More specifically, we say
a triple $(\vr_R, \vt_R, \vu_R)$ is  \emph{Riemann solution} of the Euler system in $[0,T) \times \mathbb{R} \times \mathbb{T}^1$ emanating from the Riemann data 
data $(\vr_\ell, \vt_\ell, \vu_\ell)$, $(\vr_r, \vt_r, \vu_r)$ if the following holds:
There exists $\lambda > 0$ such that 
\begin{equation} \label{i18}
	(\vr_R, \vt_R, \vu_R)(t,x) = (\vr_\ell, \vt_\ell, \vu_\ell) \ \mbox{if}\ x_1 < - \lambda t,\ 
	(\vr_R, \vt_R, \vu_R)(t,x) = (\vr_r, \vt_r, \vu_r) \ \mbox{if}\ x_1 > \lambda t;
\end{equation}
\begin{equation} \label{i19}
	\int_0^T \intTS{ \Big( \vr_R \partial_t \varphi + \vr_R \vu_R \cdot \Grad \varphi \Big) } \dt = - \intTS{ \left( \mathds{1}_{x_1 < 0} \vr_\ell + 
		\mathds{1}_{x_1 > 0} \vr_r \right)	 \varphi (0, \cdot) }
\end{equation}
for any $\varphi \in C^1_c([0,T) \times (\mathbb{R} \times \mathbb{T}^1))$; 
\begin{align} \label{i20}
	\int_0^T & \intTS{ \Big( \vr \vu \cdot \partial_t \bfphi + \vr \vu \otimes \vu: \Grad \bfphi + p(\vr, \vt) \Div \bfphi \Big) } \dt \br 
	&= - \intTS{ \left( \mathds{1}_{x_1 < 0} \vr_\ell \vu_\ell + 
		\mathds{1}_{x_1 > 0} \vr_r \vu_r	    \right)\cdot \bfphi (0, \cdot) }	
\end{align}
for any $\bfphi \in C^1_c([0,T) \times (\mathbb{R} \times \mathbb{T}^1) ; \mathbb{R}^2)$;
\begin{align} 
	\int_0^T &\intTS{ \left( \left( \frac{1}{2} \vr |\vu|^2 + \vr e(\vr, \vt) \right) \partial_t \varphi + 
		\left( \frac{1}{2} \vr |\vu|^2 + \vr e(\vr, \vt) + p(\vr, \vt) \right) \vu  \cdot \Grad \varphi \right) } \dt \br &= - \intTS{
		\left[ \mathds{1}_{x_1 < 0} \left( \frac{1}{2} \vr_\ell |\vu_\ell|^2 + \vr_\ell e(\vr_\ell, \vt_\ell) \right) + \mathds{1}_{x_1 > 0} \left( \frac{1}{2} \vr_r |\vu_r|^2 + \vr_r
		e(\vr_r, \vt_r) \right) 
		\right] \varphi (0, \cdot) }
	\label{i21}
\end{align}
for any $\varphi \in C^1_c([0,T) \times(\mathbb{R} \times \mathbb{T}^1) )$; 	
\begin{align} 	
	\int_0^T &\intTS{ \Big( \vr s(\vr, \vt) \partial_t \varphi + \vr s(\vr, \vt) \vu \cdot \Grad \varphi \Big) } \dt \br &\leq - 
	\intTS{ \left( \mathds{1}_{x_1 < 0} \vr_\ell s(\vr_\ell, \vt_\ell) +   \mathds{1}_{x_1 > 0} \vr_r s(\vr_r, \vt_r)\right) \varphi (0, \cdot) }
	\label{i22}	
\end{align}
for any  $\varphi \in C^1_c([0,T) \times (\mathbb{R} \times \mathbb{T}^1))$, $\varphi \geq 0$.

The following result was proved by Klingenberg et al \cite[Theorem 1.1]{KlKrMaMa}, see also Al Baba et al. \cite{ABKlKrMaMa}. 

\begin{Theorem} [\bf Ill posedness for Riemann problem] \label{Ti2} 
	
	\noindent	
	There exist Riemann data \\ $(\vr_\ell, \vt_\ell, \vu_\ell)$, $(\vr_r, \vt_r, \vu_r)$ such that the Riemann problem \eqref{i18}--\eqref{i22} 
	admits infinitely many solutions $(\vr_R, \vt_R, \vu_R)$ in $[0,T) \times (\mathbb{R} \times \mathbb{T}^1)$, $T > 0$ arbitrary. Moreover, all 
	solutions satisfy 
 \[
		(\vr_R, \vt_R, \vu_R)(t,x) = (\vr_\ell, \vt_\ell, \vu_\ell) \ \mbox{if}\ x_1 < - \lambda t,\ 
		(\vr_R, \vt_R, \vu_R)(t,x) = (\vr_r, \vt_r, \vu_r) \ \mbox{if}\ x_1 > \lambda t,
\]
	with the same constant $\lambda$, and there are constants 
	\[
	0 < \underline{\vt} \leq \Ov{\vt},\ 
	0 < \underline{\vr} \leq \Ov{\vr},\ \Ov{\vu} > 0 
	\]
	such that 
	\[
		\underline{\vt} \leq \vt_R \leq \Ov{\vt},\ 	\underline{\vr} \leq \vr_R \leq \Ov{\vr},\ |\vu_R| \leq \Ov{\vu}
	\]
	a.e. in $(0,T) \times (\mathbb{R} \times \mathbb{T}^1)$.	
\end{Theorem}
\noindent
Note carefully that the solutions claimed in Theorem \ref{Ti2} are entropy admissible. 

Motivated by the above result, we say that initial data $(\vr_0, \vt_0, \vu_0)$ are \emph{wild} if there exists a time $T > 0$ such that the Euler system 
admits infinitely many entropy admissible weak solutions emanating from $(\vr_0, \vt_0, \vu_0)$ in $[0,\tau] \times \mathbb{T}^d$ for any $0 < \tau < T$.
It turns out that the set of wild data is in fact dense in the $L^q$ topology for any finite $q$ \cite[Theorem 3.2]{ChFe2023}: 
\begin{Theorem}[\bf Density of wild data] \label{Ti3}
	Let 
	\[
		\vr_0 \in W^{k,2}(\TD),\ \inf_{\TD} \vr_0 > 0,\ \vt_0 \in W^{k,2}(\TD),\ \inf_{\TD} \vt_0 > 0,\ \vu_0 \in W^{k,2}(\TD; \mathbb{R}^2),\ k > 2,
	\] 
	and $1 \leq q < \infty$ be given.

	Then for any $\ep > 0$, there exist initial data $(\vr_{0, \ep}, \vt_{0, \ep}, \vu_{0, \ep})$ enjoying the following properties:
	\begin{itemize}
		\item
		The data	
		$(\vr_{0, \ep}, \vt_{0, \ep}, \vu_{0, \ep})$ are piecewise smooth, specifically, there are finitely many points $x_1^1, \dots, x^N_1$ such that
		$(\vr_{0, \ep}, \vt_{0, \ep}, \vu_{0, \ep})(x_1,x_2)$ are continuously differentiable whenever $x_1 \not \in \{ x_1^1, \dots, x^N_1 \}$, 
		$x_2 \in \mathbb{T}^1$.
		\item 
		There exists $T > 0$ such that the Euler system admits infinitely many admissible weak solutions $(\vr^n, \vt^n, \vu^n)_{n \in N}$ 
		in $L^\infty([0,T) \times \TD; \mathbb{R}^4)$ emanating from the initial data $(\vr_{0, \ep}, \vt_{0, \ep}, \vu_{0, \ep})$ such that 
		\[
			(\vr^n, \vt^n, \vu^n)|_{[0,\tau) \times B_\ep} \not \equiv (\vr^m, \vt^m, \vu^m)|_{[0,\tau) \times B_\ep}, \forall \ m \ne n	
		\]
		whenever $0 < \tau < T$ and $B_\ep \subset \TD$ is a ball of radius $\ep$.
		\item 
		\[
			\left\| \Big(\vr_{0,\ep} - \vr_0;  \vt_{0,\ep} - \vt_0; \vu_{0,\ep} - \vu_0 \Big) \right\|_{L^q(\TD; \mathbb{R}^4)} \leq \ep;
		\]
		
		\item 
		\[
			\inf_{[0,T) \times \TD} \vr^n > 0,\ \inf_{[0,T) \times \TD}	 \vt^n > 0, 
		\]
		and there exists a compact set $\mathcal{S} \subset  \TD$, $|\mathcal{S}| < \ep$ such that 
		$(\vr^n, \vt^n, \vu^n)$ are continuously differentiable in $(0,T) \times (\TD \setminus \mathcal{S})$ for any $n \in N$.
	\end{itemize}
	
\end{Theorem}

Other examples of wild data for the Euler system \eqref{i1}--\eqref{i3}, including the case of a bounded domain with an impermeable boundary, were obtained in  \cite{FeKlKrMa}.
Problems with random forcing were studied in \cite{ChFeFl2019}. 

\subsection{Well posedness for the Euler system?}

Let us finish the introductory section by formulating the first and still largely open problem:

\bigskip
\hrule
\bigskip
	
{\bf Open problem I:} 

\noindent
{\it Despite the large number of recent results concerning the existence of wild data and infinitely many weak solutions to the Euler 
system, the existence of an entropy admissible solution for arbitrary, say bounded, initial data remains open.  
Specifically, does the Euler system admit a globally defined entropy admissible weak solution for any finite energy initial data? 
}

\bigskip
\hrule
\bigskip

Introducing the class of weak solutions seemingly did not have any positive impact on the problem of well posedness of the Euler system. 
As we have seen above, even the entropy admissible ones are not uniquely determined by the (initial) data, and, moreover, 
the class of weak solutions is possibly not large enough to guarantee existence. In this survey, we summarize some recent efforts 
to restore well posedness for the Euler and possibly similar systems in fluid dynamics.

\section{Dissipative measure--valued (DMV) solutions}
\label{d}

The idea of the concept of \emph{measure--valued solution} to systems of conservation law goes back to the pioneering paper by 
DiPerna \cite{DiP2}, and was developed in the context of the incompressible Euler system in a series of papers of DiPerna and Majda \cite{DiPMaj88}, \cite{DiPMaj87}, 
\cite{DiPMaj87a}. The Euler system is interpreted as an inertial range of turbulence, specifically as a limit of a possibly oscillatory sequence of 
consistent approximations, cf. also Bardos and Titi  \cite{BarTi2013}, \cite{BaTi}.

Adopting the standard equation of state \eqref{i12}
we set
\begin{equation} \label{d1}
p(\vr, S) = (\gamma - 1) \vr e (\vr, S) = \left\{ \begin{array}{l}
\vr^\gamma \exp \left( \frac{S}{c_v \vr} \right) \ \mbox{if} \ \vr > 0, \\
0 \ \mbox{if}\ \vr = 0,\ S \leq 0, \\
\infty \ \mbox{otherwise}.
\end{array}	
\right.
\end{equation}
It is easy to check
\begin{equation} \label{d2}
	\frac{\partial \vr e(\vr, S)}{\partial S} = \vt > 0 \ \mbox{whenever}\ \vr > 0,
\end{equation}
meaning $\vr e(\vr, S)$ is a strictly increasing function of total entropy $S$.
Similarly, we define the kinetic energy
\begin{equation} \label{d3}
\frac{1}{2} \frac{|\vm|^2}{\vr} = \left\{ \begin{array}{l}
	\frac{1}{2} \frac{|\vm |^2}{\vr} \ \mbox{if}\ \vr > 0,\\
	0 \ \mbox{if}\ \vr = 0, \vm = 0. \\
	\infty \ \mbox{otherwise} \end{array}
\right.
\end{equation}
Now, by virtue of \eqref{d1}, \eqref{d3}, the total energy $E$ expressed in terms of the variables $(\vr, \vm, S)$, 
\[
E(\vr, \vm, S) = \frac{1}{2} \frac{|\vm|^2}{\vr} + \vr e(\vr, S),
\]
is a convex l.s.c. function in $R^{d+2}$ ranging in $[0, \infty]$, strictly convex on its domain. 
Note that convexity of the total energy is equivalent to   
\emph{thermodynamic stability} of the system, cf.
Bechtel, Rooney, Forest \cite{BEROFO} or
\cite[Chapter 2, Section 2.2.4]{FeLMMiSh}. This assumption
will be crucial in the subsequent analysis.

\subsection{Consistent approximations}

Very roughly indeed, one may say that a consistent approximation is a family  $(\vre, \vme, S_\ep)_{\ep > 0}$ satisfying the Euler system 
modulo consistency errors vanishing in the asymptotic limit $\ep \to 0$. 
More specifically, we call a family $(\vre, \vme, {S}_\ep)_{\ep > 0}$ is \emph{consistent approximation} of
the Euler system \eqref{i1}--\eqref{i3}, \eqref{i4} in $(0,T) \times \Omega$ with the initial data 
\[
\vr(0, \cdot) = \vr_0,\ \vm(0, \cdot) = \vm_0,\ S(0, \ep) = S_0
\]
if there holds:
	\begin{itemize}
		
		\item {\bf approximate total energy inequality:}
		\begin{equation} \label{d4}
			\intO{ E(\vre, \vme, S_\ep)(\tau, \cdot) } \leq \intO{ E(\vr_{0}, \vm_{0} , S_{0} ) } + e^1_\ep
		\end{equation}
		for a.a. $\tau \in (0,T)$, where $e^1_\ep \to 0$ as $\ep \to 0$;
		\item
		{\bf approximate equation of continuity:}
		\begin{align}
			\left[ \intO{ \vre (t, \cdot)\varphi  } \right]_{t = \tau_1-}^{t= \tau_2-} &= \int_{\tau_1}^{\tau_2} \intO{  \vme \cdot \Grad \varphi  } \dt + \int_{[\tau_1, \tau_2)} \D e^2_{\ep}[\varphi],\ \vre(0-, \cdot) \equiv \vr_{0} \br
			e^2_\ep[\varphi] &\in \mathcal{M}[0,T],\ \int_{[0,T]} \D |e^2_\ep [\varphi] | \to 0 \ \mbox{as}\ \ep \to 0
			\label{d5}
		\end{align}	
		for any $0 \leq \tau_1 \leq \tau_2 \leq T$, and any $\varphi \in C^\infty(\Ov{\Omega})$;
		
		\item {\bf approximate momentum balance:}
		\begin{align}
			\left[ \intO{ \vme (t, \cdot) \cdot \bfphi } \right]_{t = \tau_1-}^{t = \tau_2 -} &= 	
			\int_{\tau_1}^{\tau_2} \intO{ \mathds{1}_{\vre > 0} \left( \frac{\vme \otimes \vme }{\vre} + p(\vre, S_\ep) \mathbb{I} \right) :
				\Grad \bfphi  } \dt\br  &+ \int_{[\tau_1,\tau_2)} \D {e}^3_\ep [\bfphi],\ \vme(0-, \cdot) \equiv \vm_{0},
			\br
			{e}^3_\ep [\bfphi] &\in \mathcal{M}[0,T],\ \int_{[0,T]} \D |{e}^3_\ep [\bfphi] | \to 0 \ \mbox{as}\ \ep \to 0	
			\label{d6}
		\end{align}
		for any $0 \leq \tau_1 \leq \tau_2 \leq T$,	$\bfphi \in C^\infty (\Ov{\Omega}; R^d)$, $\bfphi \cdot \bm{n}|_{\partial \Omega} = 0$;
		\item
		{\bf entropy minimum principle:}
		\begin{equation} \label{d7}
			S_\ep \geq \underline{s} \vre \ \mbox{a.a. in}\ (0,T) \times \Omega
		\end{equation}
	for a certain $\underline{s} \in R$; 
	\item
		{\bf approximate entropy inequality:}
		\begin{align}
			\left[ \intO{ S_\ep (t, \cdot) \varphi } \right]_{t = \tau_1-}^{t = \tau_2-} &\geq 	
			\int_{\tau_1}^{\tau_2} \intO{ \mathds{1}_{\vre > 0} \left( S_\ep \frac{\vme}{\vre} \right) : \Grad \varphi } \dt	
			+ \int_{[\tau_1, \tau_2)} \D e^4_\ep[\varphi],\br
			S_\ep(0-, \cdot)  &\equiv S_{0},   \br
			e^4[\varphi] &\in \mathcal{M}[0,T],\ \int_{[0,T]} \D |e^4 [\varphi] | \to 0 \ \mbox{as}\ \ep \to 0
			\label{d8}
		\end{align}
		holds for any $0 \leq \tau_1 \leq \tau_2 \leq T$, and any $\varphi \in C^\infty(\Ov{\Omega})$, $\varphi \geq 0$.
	\end{itemize}
	Here, the \emph{consistency errors} $e^i[\varphi]$, $i = 2,3,4$ are signed measures on the compact time interval $[0,T]$.
	In particular, the quantities
	\[
	t \mapsto \intO{\vre (t, \cdot) \varphi},\ t \mapsto \intO{\vme \cdot \bfphi },\
	t \mapsto \intO{ S_\ep (t, \cdot) \varphi }
	\]
	can be interpreted as BV functions defined in $[0,T]$.

It is worth noting that we have replaced the energy equation \eqref{i3} by the integrated total energy balance \eqref{d4} and added the entropy inequality 
\eqref{d8}. Consistent approximations arise as solutions of suitable numerical schemes, several examples are discussed in the monograph
\cite[Part III]{FeLMMiSh}. Another, more physical example of a consistent approximation is the sequence obtained 
in the vanishing viscosity/dissipation limit, see \cite[Section 3.4]{BreFei17A}, \cite{FeiKliMar}. From this point of view, the asymptotic behaviour 
of a sequence of consistent approximation examined below is closely related to turbulence. 

\subsection{Asymptotic limits of consistent approximations}
\label{al}

It follows from the energy inequality \eqref{d4} combined with the minimum entropy
principle \eqref{d7} that
\begin{align}
	\vre &\to \vr \ \mbox{weakly-(*) in}\ L^\infty (0,T; L^\gamma(\Omega)), \br
	\vme &\to \vm \ \mbox{weakly-(*) in}\ L^\infty (0,T; L^{\frac{2 \gamma}{\gamma + 1}}(\Omega; R^d)), \br
	S_\ep &\to S \ \mbox{weakly-(*) in}\ L^\infty (0,T; L^\gamma(\Omega))
	\label{d9}
\end{align}
at least for a suitable subsequence, see \cite[Chapter 7]{FeLMMiSh}. Moreover, by the same token, we may assume
\begin{equation} \label{d10}
E(\vre, \vme, S_\ep) \to \Ov{E(\vr, \vm, S)} \ \mbox{weakly-(*) in}\ L^\infty(0,T; \mathcal{M}^+(\Ov{\Omega})), 
\end{equation}
where, in view of convexity of $E$, 
\begin{equation} \label{d11}
0 \leq \intO{ E(\vr, \vm, S) (\tau, \cdot) } \leq 
\int_{\Ov{\Omega}} \D \Ov{E(\vr, \vm, S)} \leq \intO{ E(\vr_0, \vm_0, S_0)} \equiv \mathcal{E}_0 \ \mbox{for a.a.}\ \tau \in (0,T).
\end{equation}
By analogy with the theory of (deterministic) turbulence, we may interpret $\mathcal{E}_0$ as the total (conserved) energy, 
$\intO{ E(\vr, \vm, S) }$ as the mean flow energy, and their difference 
$\mathcal{E}_0 - \intO{ E(\vr, \vm, S) }$ as turbulent energy. We say the consistency limit is \emph{turbulent} if 
$\mathcal{E}_0 - \intO{ E(\vr, \vm, S) } \ne 0$. As the weak solutions of the Euler system conserve the total energy, 
we immediately observe the following implication:
\begin{equation} \label{d12}
\mbox{{\it consistency limit is turbulent}} \ \Rightarrow \ 
\mbox{{\it the limit}}\ (\vr, \vm, S) \ \mbox{{\it IS NOT a weak solution of the Euler system.}}
\end{equation}	

\subsubsection{Dissipative solutions--limits of turbulent consistent approximations}

In view of the convergence \eqref{d9}, it is possible to associate to 
the sequence $(\vre, \vme, S_\ep)_{\ep > 0}$ a parametrised (Young) measure $( \mathcal{V}_{t,x} )_{(t,x) \in (0,T) \times \Omega}$, 
\begin{equation} \label{d13}
	\mathcal{V} = \mathcal{V}_{t,x}: (t,x) \in (0,\infty) \times \Omega \to \mathfrak{P} (\R^{d+2}),\
	\mathcal{V} \in L^\infty_{\rm weak-(*)}((0,\infty) \times \Omega; \mathfrak{P} (\R^{d+2})),
\end{equation}
where $\mathfrak{P}$ denotes the set of Borel probability measures defined
on the space of ``dummy'' variables
\[
\R^{d+2} = \left\{ (\tvr, \tvm, \widetilde{S}) \Big|  (\vr, \vm,  S) \in \R^{d+2}             \right\},
\]
\begin{equation} \label{d14}
\int_{\R^{d+2}} F(\tvr, \tvm, 	\widetilde{S}) \D \mathcal{V}_{t,x} = \left( {\rm weak-(*)}\lim_{\ep \to 0} F(\vre, \vme, S_{\ep}) \right)(t,x) 
\ \mbox{for a.a.}\ (t,x) \in (0,T) \times \Omega
\end{equation}
for any $F \in BC(\R^{d + 2})$. Note this step requires passing to a new subsequence as the case may be. 	

Possible concentrations are captured by 
the concentration measure
\begin{align} 
\mathfrak{C} &= {\rm weak-(*)} \lim_{\ep \to \infty} \left[ \frac{\vme \otimes \vme}{\vre} + p(\vre, S_\ep) \mathbb{I} \right] - 
\int_{\R^{d+2}} E(\tvr, \tvm, \widetilde{S}) \D \mathcal{V} 	\br
\mathfrak{C} &\in L^\infty_{\rm weak-(*)}(0,\infty; \mathcal{M}^+ (\Ov{\Omega}; \R^{d \times d}_{\rm sym})),
\label{d15}
\end{align}
where
the symbol $\mathcal{M}^+ (\Ov{\Omega}; \R^{d \times d}_{\rm sym}))$ denotes the set of all matrix-valued positively semi-definite finite Borel measures on
the compact set $\Ov{\Omega}$. 

The limit $(\vr, \vm, S)$, together with the total energy $\mathcal{E}_0$, is termed 
a dissipative (measure--valued) (DMV) solution of the Euler system.

\begin{Definition}[{\bf DMV solution}] \label{Dd1}

A trio $(\vr, \vm, S)$, together with the total energy $\mathcal{E}_0$ is called \emph{dissipative measure--valued (DMV) solution} of the Euler system \eqref{i1}--\eqref{i3}, \eqref{i4} in $[0,\infty) \times \Omega$, with the initial data 
$(\vr_0, \vm_0, S_0, \mathcal{E}_0)$
if there exists a parametrized measure $\{ \mathcal{V}_{t,x} \}$ and
a concentration measure $\mathfrak{C}$ such that the following holds:
\begin{itemize}
\item {\bf compatibility:}
\begin{equation} \label{d16}
(\vr , \vm, S)(t,x) = \int_{\R^{d+2}} (\tvr, \tvm, \widetilde{S})\ \D \mathcal{V}_{t,x} \ \mbox{for a.a.}\ (t,x) \in (0,T) \times \Omega;	
\end{equation}				
\item
{\bf equation of continuity:}
\begin{equation} \label{d17}
\int_0^\infty \intO{ \left[ \vr (t,x) \partial_t \varphi(t,x) +
\left( \int_{\R^{d+2}} \tvm \D \ \mathcal{V}_{t,x} \right) \cdot \Grad \varphi(t,x) \right] } \dt = - \intO{ \vr_0(x) \varphi(0, x) } 	
\end{equation}
for any $\varphi \in C^1_c([0,\infty) \times \Ov{\Omega})$;
\item
{\bf momentum equation:}
\begin{align}
\int_0^\infty &\intO{ \left[ \vm (t,x) \cdot \partial_t \bfphi(t,x)  + \left( \int_{\R^{d+2}} \mathds{1}_{\tvr > 0} \left(
\frac{ \tvm \otimes \tvm}{\tvr} + p(\tvr, \widetilde{S}) \mathbb{I} \right) \D \mathcal{V}_{t,x} \right) :\Grad \bfphi(t,x) 	 \right] } \dt \br
&= - \int_0^\infty \int_{\Ov{\Omega}} \Grad \bfphi(t,x) : \D \mathfrak{C} (t,x)
\label{d18}	- \intO{ \vc{m}_0(x) \cdot \bfphi(0, x) }
	\end{align}
for any $\bfphi \in C^1_c([0,\infty) \times \Ov{\Omega}; \R^d)$, $\bfphi \cdot \bm{n}|_{\partial \Omega} = 0$;
\item
{\bf entropy inequality:}
\begin{align}
\int_0^\infty &\intO{ \left[ S(t,x) \partial_t \varphi(t,x) + \left( \int_{\R^{d+2}}
\mathds{1}_{\tvr > 0} \widetilde{S} \frac{\tvm}{\tvr} \ \D \mathcal{V}_{t,x} \right) \cdot \Grad \varphi(t,x) \right] } \dt
\br &\leq - \intO{ S_0(x) \varphi (0, x) }
\label{d19}		
	\end{align}
for any $\varphi \in C^1_c([0,\infty) \times \Ov{\Omega})$, $\varphi \geq 0$;
\item
{\bf energy compatibility:}
\begin{align}
	\mathcal{E}_0 \geq 
	\intO{ E(\vr_0, \vm_0, S_0) }
	&\geq \intO{ \left( \int_{\R^{d+2}} E(\tvr, \tvm, \widetilde{S}) \ \D \mathcal{V}_{t,x} \right)  }	+   r(d, \gamma)\int_{\Ov{\Omega}} \ \D {\rm trace}[ \mathfrak{C}](t,x) ,\br
	r(d, \gamma) &= \min \left\{ \frac{1}{2}; \frac{d \gamma}{\gamma - 1} \right\}
	\label{d20}
\end{align}	
for a.a. $t \in (0,\infty)$.

\end{itemize}

\end{Definition}

\medskip

The concept of DMV solution was introduced in \cite{BreFei17} and later elaborated in the monograph \cite[Chapter 5]{FeLMMiSh}. Kr\" oner and Zajaczkowski \cite{KrZa}
proposed a different class of measure--valued solution to the Euler system postulating entropy \emph{equation} rather than inequality \eqref{d19}. Unfortunately,
a physically admissible weak solution with a shock that produces entropy does not fit in this class.

\subsubsection{Basic properties of DMV solutions}
\label{bps}

The reader may consult \cite[Chapter 5]{FeLMMiSh} for 
a detailed discussion of the concept of DMV solutions to various problems in fluid mechanics and their use in numerical analysis.
Here, we record their basic properties:
\begin{enumerate}
	\item {\bf Global existence.} The Euler system admits global--in--time DMV solutions for any choice of the initial data $(\vr_0, \vm_0, S_0, \mathcal{E}_0)$ belonging to the convex set 
\begin{equation} \label{d21}
X_D = \left\{ (\vr_0, \vm_0, S_0, \mathcal{E}_0) \ \Big| \ \intO{ E	(\vr_0, \vm_0, S_0) } \leq \mathcal{E}_0,\ 
S_0 \geq \underline{s} \vr_0 \right\}. 
\end{equation}	
see \cite[Proposition 3.8]{BreFeiHof19C}.
	
	\item {\bf Weak continuity.} 
	The functions $\vr$, $\vm$ belong to the class 
	\begin{equation} \label{d22}
		\vr,\ \vm \in BC_{\rm weak}([0, \infty); L^\gamma(\Omega)); 
	\end{equation}	
	the total entropy admits one-sided limits 
	\begin{equation} \label{d23}
		\lim_{t \to \tau \pm} \intO{ S (t, \cdot) \phi  } ,\
		\lim_{t \to 0+} \intO{ S (t, \cdot) \phi} 
		\ \mbox{for any} \ \phi \in L^{\gamma'}(\Omega). 
	\end{equation}	
	In particular $S(t\pm)$ can be identified with a function in $L^\gamma(\Omega)$, where we set $S(0-) = S_0$. There holds 
	\begin{equation} \label{d24}
		S(t-) \leq S(t+) \ \mbox{for any}\ t \geq 0,\ 
		\| S(t\pm) \|_{L^\gamma(\Omega)} \leq \Ov{S} \ \mbox{for all}\ 
		t \geq 0,
	\end{equation} 
	see \cite[Chapter 5, Section 5.2]{FeLMMiSh}. 
	
	\item {\bf Compatibility with strong solutions.} If 
	$(\vr, \vm, S) \in C^1([ T_1, T_2] \times \Ov{\Omega}; \R^d)$, 
	$\inf_{[T_1, T_2] \times \Ov{\Omega}} \vr > 0$, and 
	\[
	\intO{ E(\vr, \vm, S) (t, \cdot) } = \mathcal{E}_0 \ 
	\mbox{for a.a.}\ t \in (T_1, T_2), 
	\]
	then $(\vr, \vm, S)$ is a classical solution of the Euler system 
	in $[T_1, T_2]$, see \cite[Chapter 5, Theorem 5.7]{FeLMMiSh}.
	
	\item{\bf Weak--strong uniqueness.} 
	The dissipative solution coincides with the strong $W^{1,\infty}$ solution of the Euler system as long as the latter solution exists, \cite[Chapter 6, Theorem 6.2]{FeLMMiSh}.

\item {\bf Convexity and compactness.} 
Let 
\begin{align} \label{d25}
\mathcal{U}[\vr_0, \vm_0, S_0; \mathcal{E}_0] = 
&\left\{ (\vr, \vm, S) \ \Big|\ (\vr, \vm, S) 
\ \mbox{is a dissipative solution in}\ (0,\infty) \times 
\Omega \right. \br
&\mbox{emanating from the data}\ 
(\vr_0, \vm_0, S_0; \mathcal{E}_0) \in X_D \Big\}
\end{align}
denote the solution set associated to the data $(\vr_0, \vm_0, S_0; \mathcal{E}_0)$

Then the following holds:
\begin{itemize}
	\item The set $\mathcal{U}[\vr_0, \vm_0, S_0; \mathcal{E}_0]$ is non--empty and convex.
	\item The set $\mathcal{U}[\vr_0, \vm_0, S_0; \mathcal{E}_0]$ is closed
	and bounded  
	in the weighted Lebesgue space 
	\begin{equation} \label{d26}
	L^q_\omega ((0,\infty) \times \Omega; R^{d+2}) \ \mbox{with the measure} \ \D \omega = \exp(-t)\dt \times \dx 
	 \end{equation}
for any $1\leq q \leq \frac{2 \gamma}{\gamma + 1}$. 
\item The set $\mathcal{U}[\vr_0, \vm_0, S_0; \mathcal{E}_0]$ is compact in the space
\begin{equation} \label{d27}
L^q_\omega(0, \infty; W^{-\ell,2}(\Omega; R^{d+2}) ),\ \ell > 3,	
\end{equation}
for any $1\leq q < \infty$.		 
\end{itemize}
For the proof, see \cite[Chapter 5, Theorem 5.2]{FeLMMiSh}.
\end{enumerate}

\subsection{Dissipative vs. weak solutions}

As observed, the set of all DMV solutions $\mathcal{U}[\vr_0, \vm_0, S_0, \mathcal{E}_0]$ emanating from given initial data is convex and compact in the separable Banach space 
$L^q_\omega(0, \infty; W^{-\ell,2}(\Omega; R^{d+2}) )$, $\ell > 3, q > 1$ finite. 
Let
\[
(\vr, \vm, S) \in \mathcal{U}[\vr_0, \vm_0, S_0, \mathcal{E}_0].
\]
By virtue of Choquet's generalization 
of Krein--Milman theorem, see e.g. Phelps \cite{Phel}, there exists 
a Borel probability measure $\mu$ supported on the \emph{extremal points of the convex set}  $\mathcal{U}[\vr_0, \vm_0, S_0, \mathcal{E}_0]$ such that 
\begin{equation} \label{d28}
(\vr, \vm, S) = \int_{\mathcal{U}[\vr_0, \vm_0, S_0, \mathcal{E}_0]} (\tvr, \tvm, \widetilde{S}) \D \mu.
\end{equation}
A celebrated conjecture formulated originally by DiPerna \cite{DiP2} reads: 
\bigskip 
\hrule
\bigskip
\noindent
{\bf Open problem II (DiPerna' s conjecture on extremal points):} 

\noindent
{\it The extremal points of the set $\mathcal{U}[\vr_0, \vm_0, S_0, \mathcal{E}_0]$ are weak solutions solutions of the Euler system.
}
\bigskip 
\hrule
\bigskip

\begin{Remark} \label{Rd1}
	
Here, the class of weak solutions is understood in a slightly larger sense than in Section \ref{ws}. They satisfy
\[
\mathcal{V}_{t,x} = \delta_{[\vr(t,x), \vm(t,x), S(t,x)]},\ \mathfrak{C} = 0,
\] 	
together with the following version of the Euler system
\begin{align}
	\partial_t \vr + \Div \vm &= 0,\ \vr(0, \cdot) = \vr_0, \label{d29}\\
	\partial_t \vm + \Div \left( \frac{\vm \otimes \vm}{\vr} \right) +
	\Grad p(\vr, S) &= 0,\ \vm(0, \cdot) = \vm_0, \label{d30} \\
	\partial_t S + \Div \left( S \frac{\vm}{\vr} \right) & \geq 0,\ S(0+, \cdot) = S_0
	\label{d31}	\\
	\frac{\D }{\dt} \intO{ E(\vr, \vm, S) } &=0,\ \intO{E(\vr, \vm, S)(0+, \cdot)} = \intO{ E(\vr_0, \vm_0, S_0)}, \label{d32}
\end{align}
in the sense of distributions.	
	
\end{Remark}

\section{The second law of thermodynamics}
\label{L}

Having solve the problem of existence, albeit in a larger class of DMV solutions, we focus on the choice of a suitable 
admissibility conditions to restore uniqueness of DMV solutions for given initial data. We aim to solve the problem by imposing 
some sort of \emph{maximality} of the entropy production.

\subsection{DiPerna's maximality criterion}

In his pioneering work \cite{DiP2}, DiPerna introduced the following \emph{partial ordering} of global-in-time DMV solutions belonging to the set 
$\mathcal{U}[\vr_0, \vm_0, S_0, \mathcal{E}_0]$, namely  
\begin{align} 
(\vr^1, \vm^1, S^1, \mathcal{E}_0)  \prec_{\rm DiP} 	(\vr^2, \vm^2, S^2, \mathcal{E}_0) \ \Leftrightarrow \ 
\intO{ S^1(t-, \cdot) } \leq \intO{ S^2(t-, \cdot) }\ \mbox{for any}\ t > 0.
\label{L1}	
\end{align}

\bigskip 
\hrule
\bigskip
\noindent
{\bf Open problem III (DiPerna' s conjecture on maximal entropy production):} 

\noindent
{\it Any DMV solution in $\mathcal{U}[\vr_0, \vm_0, S_0, \mathcal{E}_0]$ that is \emph{maximal} with respect to the order relation 
$\prec_{\rm DiP}$ is a weak solution of the Euler system in the sense of \eqref{d29}--\eqref{d32}. 
}
\bigskip 
\hrule
\bigskip

\noindent
Although the \emph{existence} of $\prec_{\rm DiP}$-maximal solutions was shown in \cite{BreFeiHof19C}, DiPerna's conjecture remains open. 

Let us introduce the \emph{energy defect} 
\begin{equation} \label{L2}
D_E (\vr, \vm, S, \mathcal{E}_0)(t\pm) = \mathcal{E}_0 - \intO{ E(\vr ,\vm, S)(t\pm, \cdot) } \ \mbox{for}\ t > 0.
\end{equation}	
As pointed out in Section \ref{al}, DMV solutions with small energy defect are ``close'' to weak solutions. 
The following results can be therefore interpreted as asymptotic validity of DiPerna's conjecture for $t \to \infty$, see \cite[Proposition 3.3]{FeiLM2025II}.

\begin{Theorem}[\bf Asymptotic regularity of maximal solutions] \label{TL1}
	
	Let $(\vr, \vm, S)$ be a $\prec_{\rm DiP}$ maximal solution in the set $\mathcal{U}[\vr_0, \vm_0, S_0, \mathcal{E}_0]$. 
	
	Then 
	\[
	D_E(\vr, \vm, S)(t\pm) = \mathcal{E}_0 - \intO{ E(\vr, \vm, S)(t\pm, \cdot) } 
	\to 0 \ \mbox{as}\ t \to \infty.
	\]

\end{Theorem}	

The next result shows that any set $\mathcal{U}[\vr_0, \vm_0, S_0, \mathcal{E}_0]$ contains a solution with arbitrarily small energy defect, 
see \cite[Theorem 7.1]{FeiLukYu}.

\begin{Theorem}[{\bf DMV solutions with small defect}] \label{TL2}
	Let the initial data $(\vr_0, \vm_0, S_0)$ satisfy 
	\[
	\mathcal{E}_0 = \intO{ E(\vr_0, \vm_0, S_0) }.
	\]
	
	Then for any $\delta > 0$, there exists a DMV solution $(\vr, \vm, S) \in  \mathcal{U}[\vr_0, \vm_0, S_0, \mathcal{E}_0]$ 
	such that
	\[
	0 \leq D_E(t) = \intO{ E(\vr_0, \vm_0, S_0) } - \intO{ E(\vr, \vm, S)(t, \cdot) } \leq \delta \ \mbox{for a.a.}\ t \in (0,\infty).	
	\]	
	
\end{Theorem}

We point out that the proof of Theorem \ref{TL2} is non--constructive based on Zorn's lemma (Axiom of Choice). The possibility of effectively computing 
such a solution remains dubious.

\subsection{Dafermos' maximality criterion} 

It follows from the entropy inequality \eqref{d19} that the total entropy
\[
t \in [0, \infty) \to \intO{ S(t, \cdot) }
\]
can be identified with a non--decreasing (c\` agl\` ad) function, with $S(0-, \cdot) = S_0$. In particular,
\begin{align}
	0 \leq \frac{\D^+}{\dt} \intO{ S(t, \cdot) } &\leq \infty
	\ \mbox{for any}\ t \geq 0,\br
	\frac{\D^+}{\dt} \intO{ S (t, \cdot)} &= \frac{\D }{\dt} \intO{S(t, \cdot)} < \infty \ \mbox{for a.a.}\ t \in (0, \infty).
	\label{L3}
\end{align}	

Dafermos \cite{Dafer} proposed a refined admissibility criterion based on an order relation $\prec_{\rm Daf} $
between two solutions of the Euler system :
\begin{align}
	(\vr^1, \vm^1, S^1, \mathcal{E}_0) &\prec_{\rm Daf}
	(\vr^2, \vm^2, S^2)\br \ &\Leftrightarrow \br
	\mbox{there exists}\ \tau \geq 0 \ \mbox{such that}\ (\vr^1, \vm^1, S^1)(t, \cdot) &=
	(\vr^2, \vm^2, S^2)(t, \cdot) \ \mbox{for all}\
	t \leq \tau, \br
	\frac{\D^+}{\dt} \intO{ S^2 (\tau, \cdot) } &>
	\frac{\D^+}{\dt}  \intO{ S^1 (\tau, \cdot) }.
	\label{L4}
\end{align}	

In contrast with DiPerna's criterion based on \eqref{L1}, Dafermos' admissibility criterion is of local character, in particular, 
it is independent of the behaviour of the solutions after the time $\tau$. We report the following result, see \cite[Theorem 3.2]{FeiLukYu}. 

\begin{Theorem}[{\bf Regularity of maximal DMV solutions}] \label{TL3}
	
	Under the hypotheses of Theorem \ref{TL2}, any $\prec_{\rm Daf}$ - maximal DMV solution in 
	$\mathcal{U}[\vr_0, \vm_0, S_0, \mathcal{E}_0]$
	is a weak solution of the Euler system in the sense of \eqref{d29}--\eqref{d32}.
	
\end{Theorem}	

The fact that a criterion based on maximality of the entropy production selects a weak solution of the Euler system seems quite surprising. 
Unfortunately, the \emph{existence} of a $\prec_{\rm Daf}$ - maximal solution is an open problem. 

\bigskip 
\hrule
\bigskip
\noindent
{\bf Open problem IV (Existence of maximal solutions in the sense of Dafermos):} 

\noindent
{\it Does the Euler system admit a $\prec_{\rm Daf}$ - maximal DMV solution for any finite energy initial data?
}
\bigskip 
\hrule
\bigskip

The answer to Open problem IV may be negative as the following result suggests, see \cite[Corollary 3.3]{FeiLukYu}. 

\begin{Theorem} \label{TL4}
	
Let the initial date be as in Theorem \ref{TL2}.	

Then either the Euler system admits a unique weak solution in the sense specified in \eqref{d29}--\eqref{d32},  
or there
is an infinite sequence of DMV solutions such that 
\[
(\vr^n, \vm^n, S^n) \prec_{\rm Daf} (\vr^{n+1}, \vm^{n+1}, S^{n+1}) ,\ n = 1,2,\dots
\]	

\end{Theorem}

\section{Solution semigroup}
\label{S}

Our ultimate goal is to identify a \emph{solution semigroup} for the Euler system based on DMV solutions. First, we recall the definition of the 
\emph{phase space} (formerly data space): 
\begin{align} \
	X_D &= \left\{ (\vr, \vm, S, \mathcal{E}_0)     \Big| 
	(\vr, \vm, S) \ \mbox{measurable in}\ \Omega,\ 
	S \geq \underline{s} \vr ,\ 
	\intO{ E(\vr, \vm, S) } \leq \mathcal{E}_0 \right\} 
	\br &\subset W^{-\ell,2}(\Omega; \R^{d+2}) \times [0, \infty), \ \ell > d.
	\label{S1}
\end{align}
Given $\underline {s} \in \R$, the phase space $X_D$ is a non--empty convex compact subset of the Hilbert space
$W^{-\ell,2}(\Omega; \R^{d+2}) \times \R$. 

A solution semigroup is a family of mappings 
\[
\vU:  \Big( t, [\vr_0, \vm_0, S_0, \mathcal{E}_0] \Big) \in [0, \infty) \times X_D \mapsto X_D 
\]
enjoying the following properties:
\begin{itemize}
	\item
\begin{equation}\label{S2}
\vU \Big(0 , 	[\vr_0, \vm_0, S_0, \mathcal{E}_0] \Big) = [\vr_0, \vm_0, S_0, \mathcal{E}_0];
\end{equation}
\item
{\bf semigroup property:} 
\begin{equation} \label{S3}
\vU \Big(t+s, 	[\vr_0, \vm_0, S_0, \mathcal{E}_0] \Big) = \vU \Big( t, \vU \Big(s, 	[\vr_0, \vm_0, S_0, \mathcal{E}_0] \Big) \Big)
\end{equation}	
for any $t,s \geq 0$;
\item
{\bf problem compatibility:}
\begin{equation} \label{S4}
t \in [0, \infty) \mapsto \vU \Big( t, [\vr_0, \vm_0, S_0, \mathcal{E}_0] \Big) \in 
\mathcal{U} [\vr_0, \vm_0, S_0, \mathcal{E}_0] \times \{ \mathcal{E}_0 \}
\end{equation}
\item
{\bf measurability:}
\begin{equation} \label{S5}
\mbox{the mapping}\	[\vr_0, \vm_0, S_0, \mathcal{E}_0] \in X_D \mapsto \vU \Big( t, [\vr_0, \vm_0, S_0, \mathcal{E}_0] \Big) \in X_D
\end{equation}
is Borel measurable for any fixed $t \geq 0$.	
\end{itemize}

In addition, keeping in mind the interpretation of limits of consistent approximations as inertial range of turbulence we require 
\begin{equation}\label{S6}
D_E (\vr, \vm, S, \mathcal{E}_0) (t \pm ) = \mathcal{E}_0 - \intO{ E  (\vr, \vm, S)(t\pm, \cdot)  } \to 0 \ \mbox{as}\ t \to \infty
\end{equation}
whenever $\Big[ (\vr, \vm, S)(t, \cdot), \mathcal{E}_0 \Big] = \vU \Big( t, [\vr_0, \vm_0, S_0, \mathcal{E}_0] \Big)  $, $t \geq 0$. 
The physical interpretation of \eqref{S6} is vanishing of the turbulent energy for large times.
We recall that 
property \eqref{S6} is automatically satisfied if the semigroup solution if $\prec_{\rm DiP}-$maximal.

Unfortunately, we are not able to guarantee the selected semigroup to be \emph{continuous} with respect to the initial data. 
However, in view of the Borel measurability property \eqref{S5}, 
we may use an abstract version of Egorov (Lusin) theorem 
(cf. Wisniewski \cite{Wis}) to deduce the following conclusion:

Let $\mu$ be a complete Borel probability measure on $X_D$. Then for any $\ep > 0$, $\tau > 0$ there exists a closed set $K \subset X_D$, $\mu[K] > 1-\ep$ such that 
\[
\vU (\tau, \cdot)|K \to X_D \to X_D
\]
is continuous. As a matter of fact, Wisniewski \cite{Wis} proves a stronger result, namely there is a sequence of \emph{continuous} mappings 
$(\vU_\ep (\tau, \cdot))_{\ep > 0}$ such that
\[
\vU_\ep (\tau, \cdot) \to \vU(\tau, \cdot) \quad  \mu - \mbox{almost surely.}
\]

Although the definition of DMV solution involves {\it a priori} undetermined quantities, specifically the parametrized measure $\mathcal{V}_{t,x}$ 
and the concentration measure $\mathfrak{C}$, the semigroup solution depends only on the 
data $(\vr_0, \vm_0, S_0)$ and the (constant) value of the total energy $\mathcal{E}_0$. 
As the DMV solutions are in general not uniquely determined by the data, the semi-group solution must be obtained by applying a proper \emph{selection} 
process.
In the remaining part of this section, we discuss several selection processes that give rise to a semigroup solution 
enjoying  the properties \eqref{S2}--\eqref{S6}.

\subsection{A selection process by Krylov, Cardona, and Kapitanski}

We start by specifying the topology on the solution space $\mathcal{U}[\vr_0, \vm_0, S_0, \mathcal{E}_0]$. There are two natural choices:
\begin{itemize} 
\item
STRONG topology 
\[
X_{\rm STRONG} = L^q((0, \infty) \times \Omega; \R^{d + 2}),\ \D \omega = \exp(-t) \dt \times \dx,\ 1 < q \leq \frac{2 \gamma}{\gamma + 1};
\]
\item
WEAK topology represented by the weighted Hilbert space 
\[
X_{\rm WEAK} = L^2_\omega(0,\infty; W^{-\ell,2}(\Omega; \R^{d+2})),\ \ell > d;
\] 
\end{itemize}
cf. \eqref{d26}, \eqref{d27}. 

The selection strategy, proposed in the context of stochastic equations by Krylov \cite{KrylNV} and later adapted to the deterministic 
setting by Cardona and Kapitanski \cite{CorKap}, is based a set of axioms to be satisfied by a multivalued solution set 
$\mathcal{U}[\vr_0, \vm_0, S_0, \mathcal{E}_0]$, cf. \cite{BreFeiHof19C}.

\begin{itemize}
\item {\bf [A1] Compactness.} The solution set $\mathcal{U}[\vr_0, \vm_0, S_0, \mathcal{E}_0]$ is a non--empty
\emph{compact} subset of the separable Hilbert space $X_{\rm WEAK}$, cf. Section \ref{bps}, notably \eqref{d27}.  

\item {\bf [A2] Measurability.} The mapping 
\[
[\vr_0, \vm_0, S_0, \mathcal{E}_0] \in X_D \mapsto \mathcal{U}[\vr_0, \vm_0, S_0; \mathcal{E}_0] \in  {\rm comp} \Big[ X_{\rm WEAK} \Big]
\]
is Borel measurable, where ${\rm comp}$ denotes the metric space of all compact subsets of the Hilbert space $X_{\rm WEAK}$
endowed with the Hausdorff topology, see \cite[Section 4, Lemma 4.5]{BreFeiHof19C}.

\item {\bf [A3] Time--shift.} For any $T \geq 0$ and any
$(\vr, \vm, S) \in X_{\rm WEAK}$, we define
\[
\mathcal{S}_T[ 	\vr, \vm, S ](t, \cdot) =
(\vr (t + T), \vm (t + T), S (t + T) ),\ t \geq 0.
\]
We require for any
\[
(\vr, \vm, S) \in \mathcal{U}[\vr_0, \vm_0, S_0, \mathcal{E}_0]
\]
to satisfy 
\[
\mathcal{S}_T [\vr, \vm, S, \mathcal{E}]
\in \mathcal{U}[ \vr(T, \cdot),
\vm (T, \cdot), S(T-, \cdot), \mathcal{E}_0    ]
\]
for any $T \geq 0$, see \cite{BreFeiHof19C}. 

\item {\bf [A4] Concatenation.}
For
\[
(\vr^i, \vm^i, S^i) \in X_{\rm WEAK}, \ i=1,2,\ T > 0,
\]
we define
\[
(\vr^1, \vm^1, S^1) \cup_T
(\vr^2, \vm^2, S^2) = (\vr, \vm, S),
\]
\[
(\vr, \vm, S) = \left\{
\begin{array}{l}
	(\vr^1, \vm^1, S^1) (t, \cdot) \ \mbox{for}\ 0 \leq t \leq T, \\ \\
	(\vr^2, \vm^2, S^2) (t - T, \cdot) \ \mbox{for}\ t > T	
\end{array} \right.
\]

For any
\[
(\vr^1, \vm^1, S^1) \in \mathcal{U}[\vr_0, \vm_0, S_0, \mathcal{E}_0]
\]
and
\[
(\vr^2, \vm^2, S^2) \in \mathcal{U}[\vr^1(T, \cdot), \vm^1(T, \cdot), S^1(T-, \cdot), \mathcal{E}_0], 
\]
we require
\[
(\vr^1, \vm^1, S^1) \cup_T (\vr^2, \vm^2, S^2) \in
\mathcal{U}[\vr_0, \vm_0, S_0, \mathcal{E}_0],
\]
see \cite[Lemma 5.1]{FeiLukYu}.

\end{itemize}

If axioms [A1]--[A4] are satisfied by a multivalued solution mapping, the selection procedure consists of a \emph{successive} minimization of 
a countable family of cost functionals
\begin{equation} \label{S7}
\mathcal{F}_{n,m} (\vr, \vm, S) = \int_0^\infty \exp( - \lambda_n t) \left[ G_m \Big( (\vr, \vm, S)(t, \cdot) \Big) \right]    \dt, 
\end{equation}	
where $(\lambda_n)_{n=1}^\infty$ is a dense set in $(0, \infty)$, and $(G_m)_{m=1}^\infty$ a family of continuous (bounded) functionals separating 
points in $X_D$. Specifically, given $(n,m)$, we define 
a restricted solution set 
\begin{align} 
\mathcal{U}_{n,m} \circ \mathcal{U} = &\left\{ (\tvr, \tvm, \widetilde{S}) \in \mathcal{U}[\vr_0, \vm_0, S_0, \mathcal{E}_0] \ \Big|\ 
\mathcal{F}_{n,m} (\tvr, \tvm, \widetilde{S}) \leq \mathcal{F}_{n,m} (\vr, \vm, {S})
 \right. \br
&\quad \mbox{for all}\ (\vr, \vm, {S})  \in \mathcal{U}[\vr_0, \vm_0, S_0, \mathcal{E}_0] \Big\}
\label{S8} 
\end{align}

It is possible to show that the restricted solution set $\mathcal{U}_{n,m} \circ \mathcal{U}$ complies with the axioms [A1]--[A4]. Consequently, applying this 
minimization process successively for a dense family $(n_k, m_k)_{k=1}^\infty$, 
\[
\mathcal{U}_\infty [\vr_0, \vm_0, S_0, \mathcal{E}_0] = \cap_{k \geq 1} \mathcal{U}_{n_{k+1}, m_{k+1}} \circ \mathcal{U}_{n_k, m_k} \circ \dots \circ 
\mathcal{U}_{n_{1}, m_{1}} \circ \mathcal{U}_{n_k, m_k} \mathcal{U} [\vr_0, \vm_0, S_0, \mathcal{E}_0]
\]
the resulting set is a singleton yielding the desired solution semigroup. 

The above delineated procedure has been successfully applied 
in \cite{BreFeiHof19C}. The main shortcoming of the process is the fact that the selected solution may sensitively depend on the choice/order 
of the sequence $\lambda_n$ as well as the functionals $G_m$. Reducing the number of cost functionals may not only simplify the selection 
procedure but also facilitate its eventual numerical implementation. 

\subsection{Two step selection procedure based on convexity}
\label{tss}

In \cite{FeiLM2025II}, we proposed a simplified two--step selection process based on \emph{convexity} of the solution set 
$\mathcal{U}[\vr_0, \vm_0, S_0, \mathcal{E}_0]$. 

\subsubsection{First selection}

Following the strategy of \cite{BreFeiHof19C} we introduce the 
\emph{first selection functional} 
\[
\mathcal{F}_S (\vr, \vm, S) = \int_0^\infty  \exp(-t) \left( \intO{ S(t, \cdot) } \right) 
\dt 
\]
which is a bounded linear form defined on the solution set 
$\mathcal{U}[\vr_0, \vm_0, S_0, \mathcal{E}_0]$ endowed with the topology of the Banach space $X_{\rm STRONG}$.

Similarly to the preceding section, we set
\begin{align}
	\mathcal{U}_S[ \vr_0, \vm_0, S_0, \mathcal{E}_0] &\equiv	{\rm argmax} \mathcal{F}_S \Big[ \mathcal{U}[\vr_0, \vm_0, S_0; \mathcal{E}_0] \Big]\br &= \left\{ (\vr, \vm, S) \in \mathcal{U}[\vr_0, \vm_0, S_0; \mathcal{E}_0]\ \Big|\ \mathcal{F}_S (\vr, \vm, S) \geq 
	\mathcal{F}_S (\tvr, \tvm, \widetilde{S})  \right. \br 
	&\quad \quad \mbox{for any}\ (\tvr, \tvm, \widetilde{S}) \in 
	\mathcal{U}[\vr_0, \vm_0, S_0; \mathcal{E}_0]    \Big\}. 
	\label{S9}
\end{align}

As shown in \cite{FeiLM2025II}, the restricted solution set $\mathcal{U}_S[ \vr_0, \vm_0, S_0; \mathcal{E}_0]$ enjoys the following properties:
\begin{itemize}
\item 
{\bf Convexity.} The set $\mathcal{U}_S[ \vr_0, \vm_0, S_0; \mathcal{E}_0]$ is a nonempty convex set, compact in $X_{\rm WEAK}$ and closed 
in $X_{\rm STRONG}$. 
\item 
{\bf Measurability.} The set valued mapping 
\begin{equation} \label{S10}
	\mathcal{U}_S: (\vr_0, \vm_0, S_0; \mathcal{E}_0) \in X_D \mapsto 
	{\rm argmax}\, \mathcal{F}_S \Big[ \mathcal{U}[\vr_0, \vm_0, S_0; \mathcal{E}_0] \Big] \in {\rm comp} \Big[ X_{\rm WEAK} \Big]
\end{equation}	 
is Borel--measurable. Moreover, the same mapping 
\begin{align} 
	\mathcal{U}_S: (\vr_0, \vm_0, S_0; \mathcal{E}_0) \in X_D \mapsto 
	{\rm argmax}\, \mathcal{F}_S \Big[ \mathcal{U}[\vr_0, \vm_0, S_0; \mathcal{E}_0] \Big] &\in {\rm closed} \Big[ X_{\rm STRONG} \Big]
	\label{S11} 
\end{align}	
is Borel measurable with respect to the metrisable Wijsman topology on closed convex subsets of $X_{\rm STRONG}$. We recall
the convergence in the Wijsman topology: 
\[
	\mathcal{A}_n \toW \mathcal{A} \ \Leftrightarrow\
	{\rm dist}_{L^q_\omega}[y, \mathcal{A}_n] \to {\rm dist}_{L^q_\omega}[y, \mathcal{A}] \ \mbox{for any}\ y \in L^q_\omega(0,\infty; L^q(\Omega; \R^{d+2}) .
\]
\item 
{\bf Maximal dissipation rate.} All solutions belonging to the set 
$\mathcal{U}_S[ \vr_0, \vm_0, S_0, \mathcal{E}_0]$ are maximal with respect to the partial ordering $\prec_{\rm DiP}$. In particular, they satisfy \eqref{S6}.
\end{itemize}

\subsubsection{Second selection}

If the set $\mathcal{U}_S[\vr_0, \vm_0, S_0, \mathcal{E}_0]$ fails to be a singleton, 
we propose the second selection based on maximisation of the energy defect (the so--called turbulent energy). More specifically, we consider the functional
\[
\mathcal{F}_E(\vr, \vm, S) = \int_0^\infty \exp(-t) \left( \intO{ 
	E(\vr, \vm, S) (t, \cdot)}	\right) \dt.
\]
The selection procedure is then closed by setting 
\[
\mathcal{U}_{ES}(\vr_0, \vm_0, S_0, \mathcal{E}_0) = 
{\rm argmin}\, \mathcal{F}_E \Big[ \mathcal{U}_S[\vr_0, \vm_0, S_0; \mathcal{E}_0] \Big]. 
\]
As the energy functional $E$ is strictly convex in its domain, 
$\mathcal{U}_{ES}(\vr_0, \vm_0, S_0; \mathcal{E}_0)$ is a singleton, which finishes the selection process.
Intuitively, the second choice prefers ``turbulent'' solutions. A similar selection process was proposed by 
Klingenberg, Markefelder, and Wiedemann \cite{KlMaWi}.

As shown in \cite{FeiLM2025II}, the mapping 
\begin{align} 
	\mathcal{U}_{ES} : (\vr_0, \vm_0, S_0, \mathcal{E}_0) \in X_D \mapsto 
	{\rm argmin}\, \mathcal{F}_E \Big[ \mathcal{U}_S[\vr_0, \vm_0, S_0; \mathcal{E}_0]     \Big]	
	\in X_{\rm STRONG}
	\label{S12}
\end{align}
is Borel measurable. This observation is based on the construction 
of a Moreau--Yosida regularisation of the 
functional $\mathcal{F}_E$ on the Banach space 
$X_{\rm STRONG}$ proposed by 
Bacho \cite{Bacho}:
\begin{align} 
	\mathcal{F}^\ep_E &(\vr, \vm, S ) \br &= 
	\inf_{(\tvr, \tvm, \widetilde{S}) \in L^q_{\omega}((0, \infty); L^q(\Omega; \R^{d+2}))  } \left[ \frac{1}{\ep} 
	\| (\vr, \vm, S ) - (\tvr, \tvm, \widetilde{S}) \|_{L^q_{\omega}((0, \infty); L^q(\Omega; \R^{d+2})  }^q + 
	\mathcal{F}_E (\tvr, \tvm, \widetilde{S})) \right]
	\label{S13}	
\end{align}
for $\ep > 0$, see \cite{FeiLM2025II} for details. 

Finally, it is easy to see that the instantaneous values of dissipative solution belong to the set $X_D$, 
\[
\Big[ \mathcal{U}_{ES}(\vr_0, \vm_0, S_0; \mathcal{E}_0); \mathcal{E}_0 \Big] (\tau-, \cdot) \in X_D \ \ \mbox{for any}\ \tau > 0.
\]
Moreover, the entropy $S$ can be identified with a c\` agl\` ad mapping ranging in the space $W^{-\ell, 2}(\Omega)$. In particular, 
the entropy $S$ belongs to the Skorokhod space of  c\` agl\` ad mappings,
\[
S \in D([0, \infty); W^{-\ell,2}(\Omega)).
\]
We refer to \cite[Appendix A1]{FeiNovOpen} for basic properties of the space $D([0, \infty); W^{-\ell,2}(\Omega))$.

We conclude that the two-step selection process specified above yields a solution semigroup 
\[
t \mapsto \vU (t, [\vr_0, \vm_0, S_0, \mathcal{E}_0]) = \Big[ \mathcal{U}_{ES}(\vr_0, \vm_0, S_0, \mathcal{E}_0), \mathcal{E}_0 \Big]
\]
possessing all properties specified in \eqref{S2}--\eqref{S6}.

\subsection{Second Law of Thermodynamics - a single step selection}
\label{sss}

As we show below, the selection can be based on a single cost functional as long as the latter is strictly convex on its domain. 
Motivated by Clausius' statement quoted in Section \ref{se}, we propose a single step selection criterion based 
on minimization of a distance between a DMV solution and the background equilibrium state.

\subsubsection{Background equilibrium state}

It is a direct consequence of the equation of continuity \eqref{d17} that the total mass 
\begin{equation} \label{S14}
	\intO{ \vr(t, \cdot) } =  \intO{\vr_0 } = M_0 > 0 
\end{equation}
is conserved for any DMV solution. Given the total energy $\mathcal{E}_0$, the background (stable) 
equilibrium density-temperature is given by the constant state 
$(\Ov{\vr}, \Ov{\vt})$,   
\[
\Ov{\vr} = \avintO{ \vr_0 },\ \mathcal{E}_0 = c_v \intO{ 
	\Ov{\vr} \Ov{\vt} } \ \Rightarrow \ \Ov{\vt} = \frac{\mathcal{E}_0} {c_v M_0}.
\]

Next, let us introduce the Bregman distance (divergence) associated to the convex energy $E$, 
\begin{align}
	E &\left( \vr, \vm, S \Big| \tvr, \tvm, \tvS \right) \br
	&= E( \vr, \vm, S) - \frac{\partial E (\tvr, \tvm, \tvS) }{\partial \vr} (\vr - \tvr) -  \frac{\partial E (\tvr, \tvm, \tvS) }{\partial \vm} \cdot (\vm - \tvm) -  \frac{\partial E (\tvr, \tvm, \tvS) }{\partial S} (S - \tvS) \br &\quad - E(\tvr, \tvm, \tvS).
	\nonumber
\end{align}
We refer to \cite[Chapter 5]{FeLMMiSh} for a detailed discussion of the concept of Bregman divergence associated to the (convex) total energy $E$.
At this point, it is enough to observe that  $E \left( \vr, \vm, S \Big| \tvr, \tvm, \tvS \right)$ represents a ``metric'' in the sense that 
\[ 
E \left( \vr, \vm, S \Big| \tvr, \tvm, \tvS \right) \geq 0,\ E \left( \vr, \vm, S \Big| \tvr, \tvm, \tvS \right) = 0
\ \Leftrightarrow \  (\vr, \vm, S) = (\tvr, \tvm, \tvS).
\]

Evaluating the Bregman distance of a DMV solution $(\vr, \vm, S)$ to the equilibrium solution 
$(\Ov{\vr}, 0 , \Ov{\vr} s(\Ov{\vr}, \Ov{\vt}))$, and integrating over $\Omega$, we obtain 
\begin{equation} \label{S15}
	\intO{ E \left( \vr, \vm, S \Big| \Ov{\vr} , 0 , \Ov{\vr} s(\Ov{\vr}, \Ov{\vt}) \right) } = \intO{ \Big[  E(\vr, \vm, S) - \Ov{\vt} S ] } + \Ov{\vt} M_0 s(\Ov{\vr}, \Ov{\vt}) - \mathcal{E}_0,
\end{equation}	
cf. \cite[Chapter 5]{FeLMMiSh}. 
This motivates the choice of the cost functional 
\begin{equation} \label{S16}
	\mathcal{F}(\vr, \vm, S) = 
	\int_0^\infty \exp(-t) \left( \intO{ \Big[  E(\vr, \vm, S)(t, \cdot) - \Ov{\vt} S(t, \cdot)  \Big] } \right) \dt, 
	\Ov{\vt} = \frac{\mathcal{E}_0}{c_v M_0},
\end{equation}
cf. \cite{FeiLM2025II}.

Similarly to the preceding part, we say that $(\vr, \vm, S, \mathcal{E}_0)$ is \emph{admissible} DMV solution to the Euler system with the initial data 
\[
(\vr_0, \vm_0, S_0, \mathcal{E}_0) \in X_D  
\]
if
\begin{equation} \label{S17}
	(\vr, \vm, S) = {\rm arg min}_{\mathcal{U}[\vr_0, \vm_0, S_0, \mathcal{E}_0]} \left[ \int_0^\infty 
	\exp(-t) \left( \intO{ \Big[  E(\vr, \vm, S) - \Ov{\vt} S \Big] } \right)  \dt \right], 
\end{equation}	
meaning $(\vr, \vm, S) \in \mathcal{U}[\vr_0, \vm_0, S_0, \mathcal{E}_0]$ and 
\begin{align} 
	\int_0^\infty & \exp(-t) \left( \intO{ \Big[  E(\vr, \vm, S) - \Ov{\vt} S \Big] } \right) \dt \leq 
	\int_0^\infty & \exp(-t) \left( \intO{ \Big[  E(\tvr, \tvm, \tvS) - \Ov{\vt} \tvS \Big] } \right) \dt
	\label{S18}	
\end{align}	
for all $(\tvr, \tvm, \tvS) \in \mathcal{U}[\vr_0, \vm_0, S_0, \mathcal{E}_0]$.
As the function $(\vr, \vm ,S) \mapsto E(\vr, \vm, S) - \Ov{\vt} S$ is strictly convex on its domain, the admissible solution is identified in a single 
step.

Although the admissible DMV solution are (not known to be) $\prec_{\rm DiP}$ maximal, their energy defect vanishes as $t \to \infty$, 
see \cite[Proposition 3.2]{Fei2026I}.

\begin{Theorem}[{\bf Vanishing energy defect}] \label{TS1} Let $(\vr, \vm, S, \mathcal{E}_0)$ be an admissible 
	dissipative solution of the Euler system in $(0,\infty) \times \Omega$. 
	
	Then 	
	\[
	D_E(\vr, \vm, S, \mathcal{E}_0) (t \pm) \equiv \mathcal{E}_0 - \intO{ E \Big( (\vr, \vm, S)(t \pm, \cdot) \Big) } \to 0 \ \mbox{as}\ t \to \infty.
	\]
\end{Theorem}

In view of Theorem \ref{TS1}, it is easy to check that the solution semigroup formed by admissible solutions enjoys the properties 
\eqref{S2}--\eqref{S6}.

\subsubsection{Absolute entropy maximizers?}

The reader will have noticed that we have deliberately fixed $\lambda_n = 1$ in the cost functional \eqref{S7}. 
There is a legitimate question to ask what would be the result of the minimization process if we consider
\begin{equation} \label{S19}
\mathcal{F}_\lambda (\vr, \vm, S) = 
\int_0^\infty \exp(- \lambda t) \left( \intO{ \Big[  E(\vr, \vm, S)(t, \cdot) - \Ov{\vt} S(t, \cdot)  \Big] } \right) \dt,\ \lambda > 0,
\end{equation}
instead of 
\[
	\mathcal{F} (\vr, \vm, S) = 
	\int_0^\infty \exp(- t) \left( \intO{ \Big[  E(\vr, \vm, S)(t, \cdot) - \Ov{\vt} S(t, \cdot)  \Big] } \right) \dt.
\]

Apparently, minimizing \eqref{S19} would give rise to a solution semigroup sharing the properties \eqref{S2}--\eqref{S6}.
In addition, we claim the following result concerning possible uniqueness of the selected solution. 

\begin{Theorem}[\bf Absolute entropy maximizer] \label{TS2}
Suppose there exists 
\[
(\tvr, \tvm, \tvS, \mathcal{E}_0) \in \mathcal{U}[\vr_0, \vm_0, S_0, \mathcal{E}_0]
\]
such that 
\begin{equation} \label{S20}
(\tvr, \tvm, \tvS) = {\rm arg min}_{\mathcal{U}[\vr_0, \vm_0, S_0, \mathcal{E}_0]} \left[ \int_0^\infty 
\exp(- \lambda t) \left( \intO{ \Big[  E(\vr, \vm, S) - \Ov{\vt} S \Big] } \right)  \dt \right]
\end{equation}
for all $\lambda \geq \Ov{\lambda} > 0$. 

Then 
\[
D_E(\tvr, \tvm, \tvS, \mathcal{E}_0)(t+, \cdot)(t+ ) = \mathcal{E}_0 - \intO{ E \Big( (\tvr, \tvm, \tvS)(t+, \cdot) \Big) } = 0
\]
for all $t > 0$. In particular, $(\tvr, \tvm, \tvS, \mathcal{E}_0)$ is a weak solution of the Euler system in the sense specified 
in \eqref{d29}--\eqref{d32}.

\end{Theorem}

\begin{proof}

\noindent {\bf Step 1: Positive energy defect and negative jumps of the cost functional}

Let us set 
\[
\mathcal{G}(\vr, \vm, S) = \intO{ \Big[  E(\vr, \vm, S) - \Ov{\vt} S \Big] }.
\]
Consider a DMV solution $(\vr, \vm, S, \mathcal{E}_0)$. We recall 
We recall that dissipative solutions
always satisfy
\begin{equation} \label{S21}
	\intO{ E \Big( (\vr, \vm, S)(t\pm, \cdot) \Big) } \leq \mathcal{E}_0,
	\ (\vr, \vm, S)(0-, \cdot) \equiv (\vr_0, \vm_0, S_0),
\end{equation}		
\begin{equation} \label{S22}
	S(t \pm, \cdot) \geq \underline{s} \vr.
\end{equation}	
Moreover, 
it follows from \eqref{S21} and \eqref{S22} that 
\begin{align} 
	\vr(t,x) &= 0 \ \Rightarrow \ \vm(t,x) = 0 \ \mbox{for a.a.}\ x \in \Omega, \br 
	\vr(t,x) &= 0 \ \Rightarrow \ S(t\pm,x) = 0 \ \mbox{for a.a.}\ x \in \Omega
	\label{s6}
\end{align}

Suppose that
\begin{equation} \label{s9} 
	D_E(\tau+) > 0 \ \mbox{for some}\ \tau > 0.
\end{equation}
Our aim is to show that the solution $(\vr, \vm, S, \mathcal{E}_0)$ can be extended beyond 
the time $\tau$ in such a way that the functional $\mathcal{G}$ evaluated at $\tau$ experiences a negative jump proportional 
to $D_E(\tau+)$.

First, by virtue of \eqref{s6}, we evaluate the temperature 
$\vt(\tau+, \cdot)$ using the relation 
\[
S(\tau+, \cdot) = \mathds{1}_{\vr(\tau, \cdot) > 0} \vr(\tau, \cdot) \Big( c_v \log(\vt(\tau+, \cdot)) - \log(\vr(\tau, \cdot)) \Big).
\] 
	
Next, we introduce 
\[
\widehat \vt (\tau, \cdot) = (1 + \ep) \vt(\tau+, \cdot) \ \mbox{whenever}\ \vr(\tau, \cdot) > 0  
\]
where $\ep > 0$ will be fixed below. 	

Let us denote
\[
E_{\rm int}(\tau+, \cdot) = c_v \vr(\tau, \cdot) \vt(\tau+, \cdot). 
\]
the associated internal energy.
We have
\begin{equation} \label{s10}
	\widehat{ {E}}_{\rm int}(\tau, \cdot) \equiv c_v \vr(\tau, \cdot) \widehat \vt (\tau, \cdot) = E_{\rm int}(\tau+, \cdot) + \ep c_v \vr(\tau, \cdot) \vt(\tau+, \cdot), 
\end{equation}
and
\begin{equation} \label{s11}
	\widehat{S}(\tau, \cdot) \equiv c_v \vr(\tau, \cdot) \log( \widehat \vt (\tau, \cdot) ) - \vr(\tau, \cdot) \log( \vr (\tau, \cdot) ) 
	= S(\tau+, \cdot)  + c_v \vr(\tau, \cdot) \log(1 + \ep).
\end{equation}

Thus we fix $\ep > 0$ so that 
\begin{equation} \label{s12}
	\widehat{D}_E(\tau) = 
	\mathcal{E}_0 - \intO{ E \Big( (\vr, \vm, \widehat{S})(\tau, \cdot) \Big) } = 0, 
\end{equation}		
meaning 
\begin{equation} \label{s13}
	\ep c_v \intO{ \vr(\tau, \cdot) \vt(\tau+, \cdot) } = D_E(\tau+)
	\ \Rightarrow \ 
	\ep = \frac{D_E(\tau+)}{\intO{ E_{\rm int}(\tau+, \cdot)}}. 
\end{equation} 

The relations \eqref{s12}, \eqref{s13} imply that we may continue the original solution $(\vr, \vm, S, \mathcal{E}_0)$ by concatenating it at the time $\tau$ with 
a new solution $(\widehat \vr, \widehat \vm, \widehat S, \mathcal{E}_0)$ starting from the data  
\[
(\vr(\tau, \cdot), \vm(\tau, \cdot), \widehat S (\tau, \cdot), \mathcal{E}_0), 
\]
with the energy defect $\widehat{{D}}_E(\tau+ ) = 0$.

Next, we evaluate the functional $\mathcal{G}$ at the time $\tau$ for the concatenated solution:
\begin{align} 
	\mathcal{G} \Big( (\vr, \vm, \widehat {S}) (\tau+, \cdot) \Big) &=
	\mathcal{G} \Big( (\vr, \vm, {S}) (\tau+, \cdot) \Big) \br  
	&+ \ep \intO{ E_{\rm int}(\tau+, \cdot) } - 
	\Ov{\vt} c_v \intO{ \vr(\tau, \cdot) \log(1 + \ep) }
	\label{s14}
\end{align}
In addition, using \eqref{S16}, \eqref{s13} we get
\begin{align}
	\ep \intO{ E_{\rm int}(\tau+, \cdot) } &- 
	\Ov{\vt} c_v \intO{ \vr(\tau, \cdot) \log(1 + \ep) } \br 
	&= d_E(\tau) - \mathcal{E}_0 \log \left( 1 + \frac{D_E(\tau+)}{\intO{ E_{\rm int}(\tau+, \cdot)}} \right).
	\label{s15}
\end{align}	

Finally, observe that 
\[
\intO{ E_{\rm int}(\tau+, \cdot)} \leq 
\intO{ E\Big(\vr, \vm, S)(\tau+, \cdot) \Big) } = \mathcal{E}_0 - D_E(\tau+);  
\]
whence
\begin{equation} \label{s17}
	D_E(\tau+) - \mathcal{E}_0 \log \left( 1 + \frac{D_E(\tau+)}{\intO{ E_{\rm int}(\tau+, \cdot)}} \right)
	\leq d_E(\tau) - \mathcal{E}_0 \log \left( 1 + \frac{D_E(\tau+)}{\mathcal{E}_0 - D_E(\tau+) } \right).
\end{equation}
Introducing the function
\begin{equation} \label{s18}
	G: y \mapsto \mathcal{E}_0 \log \left( 1 + \frac{y}{\mathcal{E}_0 - y } \right) - y
	= \mathcal{E}_0 \log \left( \frac{\mathcal{E}_0 }{\mathcal{E}_0 - y} \right) - y = 
	\mathcal{E}_0 \log( \mathcal{E}_0 ) - \mathcal{E}_0 
	\log(\mathcal{E}_0 - y) - y
\end{equation}
we observe that
\begin{equation} \label{s19}
	G(0) = 0,\ G'(y) = \frac{\mathcal{E}_0}{\mathcal{E}_0 - y}  - 1	> 0 \ \mbox{as long as}\ y > 0.
\end{equation}
Thus we may infer
\begin{equation} \label{s20} 
	\widehat{{D}}_E   (\tau) = 0,\ 
	\mathcal{G} \Big( (\vr, \vm, \widehat{S}) (\tau, \cdot) \Big)
	\leq \mathcal{G} \Big( (\vr, \vm, {S}) (\tau+, \cdot) \Big) -
	G(D_E(\tau+))
\end{equation}
for the concatenated solution, 
where $G$ is defined in \eqref{s18}. We conclude that the functional $\mathcal{G}$ evaluated at the concatenated solution 
experiences a negative jump proportional to $D_E(\tau+)$. 

\bigskip

\noindent
{\bf Step 2: Application of Lerch's argument}

Suppose that for some $t > 0$ the energy defect $D_E(\tau+)$ of the minimizer $(\tvr, \tvm, \tvS)$ 
satisfies $D_E(\tau+) > 0$. Consider the concatenated solution constructed in Step 1,  
\[
(\vr, \vm, S) (t, \cdot) = \left\{ \begin{array}{l} (\tvr, \tvm, \tvS)(t, \cdot) \ \mbox{for}\ t < \tau, \\ \\
(\widehat \vr, \widehat	\vt, \widehat S)(t, \cdot) \ \mbox{for}\ t \geq \tau. \end{array} \right. 
\]  
We claim that 
\[
\left[ \int_0^\infty \exp(- \lambda t) \left( \intO{ \Big[  E(\vr, \vm, S) - \Ov{\vt} S \Big] } \right)  \dt \right]
< \left[ \int_0^\infty \exp(- \lambda t) \left( \intO{ \Big[  E(\tvr, \tvm, \tvS) - \Ov{\vt} \tvS \Big] } \right)  \dt \right]
\]
for all $\lambda > 0$ large enough in contrast with \eqref{S20}. Indeed, by virtue of \eqref{s20}, there exist $D > 0$, $\delta > 0$ such that 
\[
\left( \intO{ \Big[  E(\vr, \vm, S)(t+, \cdot) - \Ov{\vt} S(t+, \cdot) \Big] } \right) 
\leq \left( \intO{ \Big[  E(\tvr, \tvm,\tvS )(t+, \cdot) - \Ov{\vt} \tvS(t+, \cdot) \Big] } \right) - D 
\] 
for all $t \in [\tau, \tau + \delta]$. Consequently, 
\begin{align}
&\left[ \int_0^\infty \exp(- \lambda t) \left( \intO{ \Big[  E(\vr, \vm, S) - \Ov{\vt} S \Big] } \right)  \dt \right] \br
&\quad - \left[ \int_0^\infty \exp(- \lambda t) \left( \intO{ \Big[  E(\tvr, \tvm, \tvS) - \Ov{\vt} \tvS \Big] } \right)  \dt \right] \br 
&\quad \leq - D \int_{\tau}^{\tau + \delta} \exp(-\lambda t) \dt + C(\mathcal{E}_0) \int_{\tau + \delta}^\infty \exp(-\lambda t) \dt \br 
&\quad \leq - D \delta \exp(-\lambda (\tau + \delta)) + \frac{C(\mathcal{E}_0)}{\lambda} \exp(-\lambda (\tau + \delta)) < 0
\nonumber	
\end{align}	
for $\lambda$ large enough.
\end{proof}	

The (hypothetical) solution $(\tvr, \tvm, \tvS, \mathcal{E}_0)$ satisfying the hypothesis of Theorem \ref{TS2} can be seen as \emph{absolute entropy maximiser}. 

\bigskip 
\hrule
\bigskip
\noindent
{\bf Open problem V:} 

\noindent
{\it Characterize the class of initial data for which the solution set $\mathcal{U}[\vr_0, \vm_0, S_0, \mathcal{E}_0]$ contains an absolute entropy 
	maximiser.
}
\bigskip 
\hrule
\bigskip

\subsection{Solution semigroup - conclusion}
\label{C}

The above discussion revealed two possible rather different approaches to the problem of well posedness for the Euler system. On the one hand, 
the Euler system can be seen as an inertial range of \emph{turbulence}. Admissible solutions are identified with limits of consistent approximations exhibiting 
oscillatory behaviour and/or possible concentrations. They are ``truly'' measure valued, whereas their energy exhibits {\it priori} undetermined turbulent component. 
On the other hand, the entropy maximization in the spirit of Dafermos or DiPerna seems to prefer the standard weak solutions. The question 
which approach yields a ``correct'' admissible solutions should be possibly decided by experiments rather than by mathematical arguments.

\subsubsection{Turbulent scenarios}

In the context of turbulent DMV solutions, the open question remains what is the origin of the underlying mechanism creating oscillations and/or concentrations 
in the sequence of consistent approximation. One possible scenario, supported by the numerical analysis performed by the group around Fjordholm, Mishra, Tadmor et al.
\cite{FjKaMiTa}, \cite{FjLyMiWe}, \cite{FjMiTa1} (see also \cite{FeLMMiSh}), is a generation of the associated Yong measure directly by oscillations in the 
approximate sequence. An alternative, suggested again by numerical experiments performed by Elling \cite{ELLI2}, \cite{ELLI3}, \cite{ELLI1}, is a strong 
(non--oscillatory) convergence of consistent approximations to a \emph{connected set} of solutions of the Euler system. Apparently, the second scenario 
is in line with the existence of infinitely many solutions revealed recently by the method of convex integration.

\subsubsection{Reducing the set of eligible DMV solutions, computable solutions}

It is clear that \emph{convexity} of the set of all DMV solutions emanating from given initial data plays a crucial role in the selection process.  
To reduce the solution set, it would be desirable to focus only on \emph{computable} ones, meaning those that can be identified as limits 
of consistent approximations. 

\bigskip
\hrule
\bigskip
\noindent
{\bf Open problem VI:} 

\noindent
{\it Let the initial data $(\vr_0, \vm_0, S_0, \mathcal{E}_0)$ be given. 
Is the set of (weak) limits of consistent approximations convex? Does the set $\mathcal{U}[\vr_0, \vm_0, S_0, \mathcal{E}_0]$ of all DMV solutions coincide 
with the set of all limits of consistent approximations?
}
\bigskip 
\hrule
\bigskip

\noindent
An (affirmative) answer to the above questions in the context of the incompressible Euler system was given by Sz{\'e}kelyhidi, and Wiedemann 
\cite{SzeWie}.

Another possibility of reducing the set of relevant DMV solutions would be to take into account only \emph{statistical limits} of consistent 
approximations. More specifically, we would consider strong limits of Ces\` aro (empirical) averages 
\begin{align}
\vr &= \lim_{N \to \infty} \frac{1}{N} \sum_{n=1}^N \vr_n, \br
\vm &= \lim_{N \to \infty} \frac{1}{N} \sum_{n=1}^N \vm_n, \br
S &= \lim_{N \to \infty} \frac{1}{N} \sum_{n=1}^N S_n
\label{C1}
\end{align}	
of sequences of consistent approximation $(\vr_n, \vm_n, S_n)_{n=1}^\infty$. In accordance with the celebrated version 
of the Banach-Saks theorem by Koml\'{o}s \cite{Kom}, any sequence $(v_n)_{n=1}^\infty$ bounded in $L^1((0,T) \times \Omega)$ 
admits a subsequence $(v_{n_k})_{k=1}^\infty$ such that there holds 
\begin{equation} \label{C2}
\lim_{N \to \infty} \frac{1}{N} \sum_{k=1}^N v_{n_k} \to v \ \mbox{a.a. in}\ (0,T) \times \Omega.
\end{equation}	
As a matter of fact, the convergence \eqref{C2} remains valid also for any subsequence of $(v_{n_k})_{k=1}^\infty$. The limit in 
\eqref{C2} coincides with the so-called biting limit on the sequence $(v_{n_k})_{k=1}^\infty$.
 
The theory of parametrized (Young) measures based on Koml\' os theorem was elaborated by Balder  \cite{Bald}, \cite{BALDER}.
Converting weak convergence to strong has been applied in the numerical analysis of the specific form of Young measures in \cite{FeiLMMiz}.

Inspired by \eqref{C1}, we may introduce a new class of \emph{computable solutions} $\mathcal{U}_{\rm comp}[\vr_0, \vm_0, S_0, \mathcal{E}_0]$, 
\begin{align} 
(\vr, \vm, S, \mathcal{E}_0) &\in \mathcal{U}_{\rm comp}[\vr_0, \vm_0, S_0, \mathcal{E}_0] \br &\Leftrightarrow \br 
\vr = \lim_{N \to \infty} \frac{1}{N} \sum_{n=1}^N \vr_n, \ \vm &= \lim_{N \to \infty} \frac{1}{N} \sum_{n=1}^N \vm_n, \
S = \lim_{N \to \infty} \frac{1}{N} \sum_{n=1}^N S_n \ \mbox{a.a. in}\ (0, \infty) \times \Omega, \br 
\ \mbox{for some sequence}\ &(\vr_n, \vm_n, S_n, \mathcal{E}_0) \ \mbox{of consistent approximations of the Euler system.} 
\label{C3} 
\end{align}

It is not difficult to show the following properties of the set $\mathcal{U}_{\rm comp}[\vr_0, \vm_0, S_0, \mathcal{E}_0]$: 
\begin{itemize}
\item {\bf convexity:}
for any initial data $(\vr_0, \vm_0, S_0, \mathcal{E}_0)$, the set of computable solutions \\
$\mathcal{U}_{\rm comp}[\vr_0, \vm_0, S_0, \mathcal{E}_0]$ is a non-empty closed convex subset of $\mathcal{U}[\vr_0, \vm_0, S_0, \mathcal{E}_0]$;
\item {\bf compatibility} 
any weak solution of the Euler system is computable.
\end{itemize}	
In view of convexity of the set $\mathcal{U}_{\rm comp}[\vr_0, \vm_0, S_0, \mathcal{E}_0]$, the selection procedures specified in Sections 
\ref{tss}, \ref{sss} may be directly applied to the set of computable solutions yielding a solution semigroup. Our last open problem reads:

\bigskip
\hrule
\bigskip
\noindent
{\bf Open problem VII:} 

\noindent
{\it We have 
\[
\mathcal{U}_{\rm comp}[\vr_0, \vm_0, S_0, \mathcal{E}_0] \subset \mathcal{U}[\vr_0, \vm_0, S_0, \mathcal{E}_0]. 
\]	
Is it true that 
\[
\mathcal{U}_{\rm comp}[\vr_0, \vm_0, S_0, \mathcal{E}_0] = \mathcal{U}[\vr_0, \vm_0, S_0, \mathcal{E}_0]
\]	
at least for a specific class of initial data?
}
\bigskip 
\hrule
\bigskip

Needless to say that considering empirical averages of numerical approximations is strongly reminiscent of the Monte Carlo method, where, however, the average is taken 
over the set of \emph{statistical solutions} emerging from i.i.d. initial data. One could therefore conjecture that non--uniqueness
of the limits generated by consistent (numerical) approximations could be created by certain ``random'' errors during the computation process.

\def\cprime{$'$} \def\ocirc#1{\ifmmode\setbox0=\hbox{$#1$}\dimen0=\ht0
	\advance\dimen0 by1pt\rlap{\hbox to\wd0{\hss\raise\dimen0
			\hbox{\hskip.2em$\scriptscriptstyle\circ$}\hss}}#1\else {\accent"17 #1}\fi}


\end{document}